\documentclass[dvipdfmx]{article}
\usepackage{amsthm, amscd, amsmath, amsfonts, mathrsfs}
\newtheorem{theo}{Theorem}[section]
\newtheorem{proposition}[theo]{Proposition}

\newtheorem{rem}[theo]{Remark}
\newtheorem{definition}[theo]{Definition}

\input xy
\xyoption{all}

\begin{document}

\title{The Bijectivity of Mirror Functors on Tori}

\author{Kazushi Kobayashi\footnote{Department of Mathematics, Graduate School of Science, Osaka University, Toyonaka, Osaka, 560-0043, Japan. E-mail : k-kobayashi@cr.math.sci.osaka-u.ac.jp. 2010 Mathematics Subject Classification : 14J33 (primary), 14F05, 53D37 (secondary). Keywords : torus, homological mirror symmetry, SYZ transform.}}

\date{}

\maketitle

\begin{abstract}
By the SYZ construction, a mirror pair $(X,\check{X})$ of a complex torus $X$ and a mirror partner $\check{X}$ of the complex torus $X$ is described as the special Lagrangian torus fibrations $X\rightarrow B$ and $\check{X}\rightarrow B$ on the same base space $B$. Then, by the SYZ transform, we can construct a simple projectively flat bundle on $X$ from each affine Lagrangian multi section of $\check{X}\rightarrow B$ with a unitary local system along it. However, there are ambiguities of the choices of transition functions of it, and this causes difficulties when we try to construct a functor between the symplectic geometric category and the complex geometric category. In this paper, we prove that there exists a bijection between the set of the isomorphism classes of their objects by solving this problem.
\end{abstract}

\tableofcontents

\section{Introduction}
Let $X$ be an $n$-dimensional complex torus, and we denote by $\check{X}$ a mirror partner of the complex torus $X$. For this mirror pair $(X,\check{X})$, the homological mirror symmetry conjecture \cite{Kon}, which is proposed by Kontsevich in 1994, states that there exists an equivalence
\begin{equation*}
D^b(Coh(X)) \cong Tr(Fuk(\check{X}))
\end{equation*}
as triangulated categories. Here, $D^b(Coh(X))$ is the bounded derived category of coherent sheaves on $X$, and $Tr(Fuk(\check{X}))$ is the derived category of the Fukaya category $Fuk(\check{X})$ on $\check{X}$ \cite{Fukaya category} obtained by the Bondal-Kapranov-Kontsevich construction \cite{bondal}, \cite{Kon}. Historically this conjecture has been first studied when $(X,\check{X})$ is a pair of elliptic curves (see \cite{elliptic}, \cite{A-inf}, \cite{abouzaid} etc.), and after that, when $(X, \check{X})$ is a pair of abelian varieties of higher dimension generalizing the case of elliptic curves \cite{Fuk} (see also \cite{dg}). 

On the other hand, the SYZ construction \cite{SYZ} by Strominger, Yau, and Zaslow in 1996, proposes a way of constructing mirror pairs geometrically. By this construction, the mirror pair $(X, \check{X})$ is realized as the trivial special Lagrangian torus fibrations $\pi : X\rightarrow B$ and $\check{\pi } : \check{X}\rightarrow B$ on the same base space $B$ which is homeomorphic to an $n$-dimensional real torus. Here, for each point $b\in B$, the special Lagrangian torus fibers $\pi ^{-1}(b)$ and $\check{\pi }^{-1}(b)$ are related by the T-duality. In particular, it is expected that the homological mirror symmetry on the mirror pair $(X,\check{X})$ is realized by the SYZ transform (an analogue of the Fourier-Mukai transform) along the special Lagrangian torus fibers of $\pi : X\rightarrow B$ and $\check{\pi } : \check{X}\rightarrow B$.
 
Concerning the above discussions, we explain the purpose of this paper. For a given mirror pair $(X,\check{X})$, we regard it as the trivial special Lagrangian torus fibrations $\pi : X\rightarrow B$ and $\check{\pi } : \check{X}\rightarrow B$ in the sense of the SYZ construction. First, in the symplectic geometry side, we consider the full subcategory $Fuk_{\mathrm{aff}}(\check{X})$ of the Fukaya category $Fuk(\check{X})$ consisting of affine Lagrangian multi sections of $\check{\pi } : \check{X}\rightarrow B$ with unitary local systems along them (in this paper, we sometimes call $Fuk_{\mathrm{aff}}(\check{X})$ simply the Fukaya category). Then, according to the discussions in \cite{leung} and \cite{A-P}, we can obtain a holomorphic vector bundle on $X$ which admits a constant curvature connection from each object of $Fuk_{\mathrm{aff}}(\check{X})$. This is called the SYZ transform. More precisely, the above constant curvature is expressed locally as
\begin{equation}
dz^t R d\bar{z}\cdot \mathrm{id}, \label{constcurv}
\end{equation}
where $z=(z_1,\cdots, z_n)^t$ is the local complex coordinates of $X$, and $R$ is a constant matrix of order $n$ (actually, $R$ is a Hermitian matrix of order $n$). On the other hand, for a holomorphic vector bundle on $X$ with the Hermitian connection, if its curvature form is expressed locally as the form (\ref{constcurv}), such a holomorphic vector bundle admits a projectively flat structure (for example, see \cite{koba}). Therefore, we see that each object of $Fuk_{\mathrm{aff}}(\check{X})$ is transformed to a projectively flat bundle on $X$, which in particular becomes simple. However, there are ambiguities of the choices of transition functions of it. In this paper, we consider a DG-category $DG_X$ consisting of such simple projectively flat bundles with any compatible transition functions. We expect that this $DG_X$ generates $D^b(Coh(X))$ though we do not discuss it in this paper. At least, it is known that it split-generates $D^b(Coh(X))$ when $X$ is an abelian variety (cf. \cite{orlov}, \cite{abouzaid}). In this setting, when we fix a choice of transition functions of holomorphic vector bundles in $DG_X$, we obtain a map
\begin{equation*}
\iota  : \mathrm{Ob}(Fuk_{\mathrm{aff}}(\check{X}))\rightarrow \mathrm{Ob}(DG_X)
\end{equation*}
by the SYZ transform. Then, for example, it is also shown in \cite[Proposition 13.26]{Fuk} that the map $\iota $ induces an injection
\begin{equation*}
\iota ^{isom} : \mathrm{Ob}^{isom}(Fuk_{\mathrm{aff}}(\check{X}))\rightarrow \mathrm{Ob}^{isom}(DG_X),
\end{equation*}
where $\mathrm{Ob}^{isom}(DG_X)$ and $\mathrm{Ob}^{isom}(Fuk_{\mathrm{aff}}(\check{X}))$ denote the set of the isomorphism classes of holomorphic vector bundles in $DG_X$ and the set of the isomorphism classes of objects of $Fuk_{\mathrm{aff}}(\check{X})$, respectively. In the present paper, we prove that the map $\iota ^{isom}$ is actually a bijection by constructing a natural map
\begin{equation*}
F : \mathrm{Ob}(DG_X)\rightarrow \mathrm{Ob}(Fuk_{\mathrm{aff}}(\check{X}))
\end{equation*}
whose direction is opposite to the direction of the map $\iota $. 

Of course, though we can also regard this result as the first step to prove the homological mirror symmetry conjecture on $(X,\check{X})$, it is a stronger statement in the following sense. In general, for two $A_{\infty }$-categories $\mathscr{C}$ and $\mathscr{C}'$, if there exists an $A_{\infty }$-equivalence $\mathscr{C}\cong \mathscr{C}'$, then it is known that there exists an equivalence $Tr(\mathscr{C})\cong Tr(\mathscr{C}')$ as triangulated categories. Hence, in order to prove the homological mirror symmetry conjecture 
\begin{equation*}
Tr(DG_X)\cong Tr(Fuk_{\mathrm{aff}}(\check{X}))
\end{equation*}
on $(X,\check{X})$, it is enough to prove that there exists an $A_{\infty }$-equivalence
\begin{equation*}
\mathscr{C}_X\cong \mathscr{C}_{\check{X}}
\end{equation*}
for some full subcategories $\mathscr{C}_X\subset DG_X$, $\mathscr{C}_{\check{X}}\subset Fuk_{\mathrm{aff}}(\check{X})$ which generate the respective triangulated categories $Tr(\mathscr{C}_X)\cong Tr(DG_X)$, $Tr(\mathscr{C}_{\check{X}})\cong Tr(Fuk_{\mathrm{aff}}(\check{X}))$ (as such categories $\mathscr{C}_X$ and $\mathscr{C}_{\check{X}}$, we would like to take the smallest possible category on each side, see also \cite{dg}). Thus, from the viewpoint of the proof of the homological mirror symmetry conjecture on $(X, \check{X})$, if we can construct a bijection $\mathrm{Ob}^{isom}(\mathscr{C}_X)\rightarrow \mathrm{Ob}^{isom}(\mathscr{C}_{\check{X}})$ and prove the existence of an $A_{\infty}$-equivalence $\mathscr{C}_X\cong \mathscr{C}_{\check{X}}$, then the map $\iota ^{isom}$ itself need not be bijective. Actually, our original motivation of proving the bijectivity of the map $\iota ^{isom}$ is that we are interested in the DG-category $DG_X$ itself and the corresponding geometric objects in the mirror dual side. In particular, in \cite{D}, exact triangles in $Tr(DG_X)$ consisting of objects of $DG_X$ are studied by using the map $F$ (see also \cite{kazushi}). Although \cite{kazushi2} is closely related to \cite{D}, we are going to revise \cite{kazushi2} by using the main result in the present paper. Furthermore, the map $F$ is also employed in discussions in \cite{kazushi3}.

This paper is organized as follows. In section 2, we take a complex torus $X$, and explain the definition of a mirror partner $\check{X}$ of the complex torus $X$. In section 3, we define a class of a certain kind of simple projectively flat bundles on $X$, and construct the DG-category $DG_X$ consisting of those holomorphic vector bundles. We also study the isomorphism classes of them in section 3. In section 4, we consider the Fukaya category $Fuk_{\mathrm{aff}}(\check{X})$ consisting of affine Lagrangian multi sections of $\check{\pi } : \check{X}\rightarrow B$ with unitary local systems along them, and study the isomorphism classes of objects of $Fuk_{\mathrm{aff}}(\check{X})$. In section 5, we explicitly construct the bijection $\mathrm{Ob}^{isom}(DG_X)\rightarrow \mathrm{Ob}^{isom}(Fuk_{\mathrm{aff}}(\check{X}))$. This result is given in Theorem \ref{bijectivity}.

\section{Preparations}
In this section, we define a complex torus $T^{2n}_{J=T}$ and its mirror partner $\check{T}^{2n}_{J=T}$.

First, we define an $n$-dimensional complex torus $T^{2n}_{J=T}$ as follows. Let $T$ be a complex matrix of order $n$ such that $\mathrm{Im}T$ is positive definite. We consider the lattice $2\pi (\mathbb{Z}^n\oplus T\mathbb{Z}^n)$ in $\mathbb{C}^n$ and define
\begin{equation*}
T^{2n}_{J=T}:=\mathbb{C}^n/2\pi (\mathbb{Z}^n\oplus T\mathbb{Z}^n).
\end{equation*}
Sometimes we regard the $n$-dimensional complex torus $T^{2n}_{J=T}$ as a $2n$-dimensional real torus $\mathbb{R}^{2n}/2\pi \mathbb{Z}^{2n}$. In this paper, we further assume that $T$ is a non-singular matrix. Actually, in our setting described bellow, it turns out that the mirror partner of the complex torus $T^{2n}_{J=T}$ does not exist if $\mathrm{det}T=0$. However, we can avoid this problem and discuss the homological mirror symmetry even if $\mathrm{det}T=0$ by modifying the definition of the mirror partner of the complex torus $T^{2n}_{J=T}$ and the class of holomorphic vector bundles which we treat. This case is discussed in \cite{kazushi3}. Here, we fix an $\varepsilon >0$ small enough and let
\begin{flalign*}
&O^{l_1\cdots l_n}_{m_1\cdots m_n}:=\biggl\{ \left(\begin{array}{ccc}x\\y\end{array}\right)\in T^{2n}_{J=T} \ | \ \frac{2}{3}\pi (l_j-1)-\varepsilon <x_j<\frac{2}{3}\pi l_j+\varepsilon , &\\
&\hspace{41mm}\frac{2}{3}\pi (m_k-1)-\varepsilon <y_k<\frac{2}{3}\pi m_k+\varepsilon, \ j,k=1,\cdots, n \biggr\}&
\end{flalign*}
be subsets of $T^{2n}_{J=T}$, where $l_j$, $m_k=1,2,3$, 
\begin{equation*}
x:=(x_1,\cdots, x_n)^t,\ y:=(y_1,\cdots, y_n)^t,
\end{equation*}
and we identify $x_i\sim x_i+2\pi $, $y_i\sim y_i+2\pi $ for each $i=1,\cdots, n$. Sometimes we denote $O^{l_1\cdots (l_j =l) \cdots l_n}_{m_1\cdots (m_k=m) \cdots m_n}$ instead of $O^{l_1\cdots l_n}_{m_1\cdots m_n}$ in order to specify the values $l_j=l$, $m_k=m$. Then, $\{ O^{l_1\cdots l_n}_{m_1\cdots m_n} \}_{l_j, m_k=1,2,3}$ is an open cover of $T^{2n}_{J=T}$, and we define the local coordinates of $O^{l_1\cdots l_n}_{m_1\cdots m_n}$ by
\begin{equation*}
(x_1,\cdots, x_n, y_1,\cdots, y_n)^t\in \mathbb{R}^{2n}.
\end{equation*}
Furthermore, we locally express the complex coordinates $z:=(z_1,\cdots, z_n)^t$ of $T^{2n}_{J=T}$ by $z=x+Ty$.

Next, we define a mirror partner $\check{T}^{2n}_{J=T}$ of $T^{2n}_{J=T}$. From the viewpoint of the SYZ construction, we should describe $\check{T}^{2n}_{J=T}$ as the trivial special Lagrangian torus fibration. However, here, instead of constructing it, we consider the $2n$-dimensional standard real torus $T^{2n}=\mathbb{R}^{2n}/2\pi \mathbb{Z}^{2n}$ with a modified (non standard) symplectic form. For each point $(x^1,\cdots, x^n, y^1,\cdots, y^n)^t \in T^{2n}$, we identify $x^i\sim x^i+2\pi $, $y^i\sim y^i+2\pi $, where $i=1,\cdots, n$. We also denote by $(x^1,\cdots, x^n, y^1,\cdots, y^n)^t$ the local coordinates in the neighborhood of an arbitrary point $(x^1,\cdots, x^n, y^1,\cdots, y^n)^t\in T^{2n}$. Furthermore, we use the same notation $(x^1,\cdots, x^n, y^1,\cdots, y^n)^t$ when we denote the coordinates of the covering space $\mathbb{R}^{2n}$ of $T^{2n}$. For simplicity, we set
\begin{equation*}
\check{x}:=(x^1,\cdots, x^n)^t, \ \check{y}:=(y^1,\cdots, y^n)^t.
\end{equation*}
We define a complexified symplectic form $\tilde{\omega } $ on $T^{2n}$ by
\begin{equation*}
\tilde{\omega } :=d\check{x}^t (-T^{-1})^t d\check{y},
\end{equation*}
where $d\check{x}:=(dx^1,\cdots, dx^n)^t$ and $d\check{y}:=(dy^1,\cdots, dy^n)^t$. We decompose $\tilde{\omega }$ into
\begin{equation*}
\tilde{\omega } =d\check{x}^t \mathrm{Re}(-T^{-1})^t d\check{y} +\mathbf{i} d\check{x}^t \mathrm{Im}(-T^{-1})^t d\check{y},
\end{equation*}
and define
\begin{equation*}
\omega :=\mathrm{Im}(-T^{-1})^t, \ B:=\mathrm{Re}(-T^{-1})^t.
\end{equation*}
Here, $\mathbf{i}=\sqrt{-1}$. Sometimes we identify the matrices $\omega $ and $B$ with the 2-forms $d\check{x}^t\omega d\check{y}$ and $d\check{x}^tBd\check{y}$, respectively. Then, $\omega $ gives a symplectic form on $T^{2n}$. The closed 2-form $B$ is often called the B-field. This complexified symplectic torus $(T^{2n},\tilde{\omega } =d\check{x}^t(-T^{-1})^td\check{y})$ is a mirror partner of the complex torus $T^{2n}_{J=T}$. Hereafter, we denote
\begin{equation*}
\check{T}^{2n}_{J=T}:=(T^{2n},\tilde{\omega } =d\check{x}^t(-T^{-1})^td\check{y})
\end{equation*}
for simplicity.

\section{Complex geometry side}
The purpose of this section is to explain the complex geometry side in the homological mirror symmetry setting on $(T^{2n}_{J=T},\check{T}^{2n}_{J=T})$. In subsection 3.1, we define a class of holomorphic vector bundles 
\begin{equation*}
E_{(r,A,r',\mathcal{U},p,q)}\rightarrow T^{2n}_{J=T}, 
\end{equation*}
and construct a DG-category 
\begin{equation*}
DG_{T^{2n}_{J=T}}
\end{equation*}
consisting of these holomorphic vector bundles $E_{(r,A,r',\mathcal{U},p,q)}$. In particular, we first construct $E_{(r,A,r',\mathcal{U},p,q)}$ as a complex vector bundle, and then discuss when it becomes a holomorphic vector bundle in Proposition \ref{prophol}. In subsection 3.2, we study the isomorphism classes of holomorphic vector bundles $E_{(r,A,r',\mathcal{U},p,q)}$ by using the classification result of simple projectively flat bundles on complex tori by Matsushima \cite{matsu} and Mukai \cite{mukai}.

\subsection{The definition of $E_{(r,A,r',\mathcal{U},p,q)}$}

We assume $r\in \mathbb{N}$, $A=(a_{ij})\in M(n;\mathbb{Z})$, and $p$, $q\in \mathbb{R}^n$. First, we define $r'\in \mathbb{N}$ by using a given pair $(r,A)\in \mathbb{N}\times M(n;\mathbb{Z})$ as follows. By the theory of elementary divisors, there exist two matrices $\mathcal{A}$, $\mathcal{B}\in GL(n;\mathbb{Z})$ such that 
\begin{equation}
\mathcal{A}A\mathcal{B}=\left( \begin{array}{cccccc} \tilde{a_1} & & & & & \\ & \ddots & & & & \\ & & \tilde{a_s} & & & \\ & & & 0 & & \\ & & & & \ddots & \\ & & & & & 0 \end{array} \right), \label{matAB}
\end{equation}
where $\tilde{a_i}\in \mathbb{N}$ ($i=1,\cdots, s, 1\leq s\leq n$) and $\tilde{a_i}|\tilde{a_{i+1}}$ ($i=1,\cdots, s-1$). Then, we define $r_i'\in \mathbb{N}$ and $a_i'\in \mathbb{Z}$ ($i=1,\cdots, s$) by
\begin{equation*}
\frac{\tilde{a_i}}{r}=\frac{a_i'}{r_i'}, \ gcd(r_i',a_i')=1,
\end{equation*}
where $gcd(m,n)>0$ denotes the greatest common divisor of $m$, $n\in \mathbb{Z}$. By using these, we set
\begin{equation}
r':=r_1'\cdots r_s'\in \mathbb{N}. \label{r'}
\end{equation}
This $r'\in \mathbb{N}$ is uniquely defined by a given pair $(r,A)\in \mathbb{N}\times M(n;\mathbb{Z})$, and it is actually the rank of $E_{(r,A,r',\mathcal{U},p,q)}$. Now, we define the transition functions of $E_{(r,A,r',\mathcal{U},p,q)}$ as follows (although the following notations are complicated, roughly speaking, the transition functions of $E_{(r,A,r',\mathcal{U},p,q)}$ in the cases of $x_j\mapsto x_j+2\pi $, $y_k\mapsto y_k+2\pi $ are given by $e^{\frac{\mathbf{i}}{r}a_jy}V_j$, $U_k$, respectively, where $j$, $k=1,\cdots,n$). Let
\begin{equation*}
\psi ^{l_1 \cdots l_n}_{m_1 \cdots m_n} : O^{l_1 \cdots l_n}_{m_1 \cdots m_n}\rightarrow O^{l_1 \cdots l_n}_{m_1 \cdots m_n}\times \mathbb{C}^{r'}, \hspace{5mm}l_j, m_k =1,2,3
\end{equation*}
be a smooth section of $E_{(r,A,r',\mathcal{U},p,q)}|_{O^{l_1 \cdots l_n}_{m_1 \cdots m_n}}$. The transition functions of $E_{(r,A,r',\mathcal{U},p,q)}$ are non-trivial on 
\begin{align*}
&O^{(l_1=3) \cdots l_n}_{m_1 \cdots m_n}\cap O^{(l_1=1) \cdots l_n}_{m_1 \cdots m_n}, \ O^{l_1 (l_2=3) \cdots l_n}_{m_1 \cdots m_n}\cap O^{l_1 (l_2=1) \cdots l_n}_{m_1 \cdots m_n},\cdots, O^{l_1 \cdots (l_n=3)}_{m_1 \cdots m_n}\cap O^{l_1 \cdots (l_n=1)}_{m_1 \cdots m_n},\\
&O^{l_1 \cdots l_n}_{(m_1=3) \cdots m_n}\cap O^{l_1 \cdots l_n}_{(m_1=1) \cdots m_n}, \ O^{l_1 \cdots l_n}_{m_1 (m_2=3) \cdots m_n}\cap O^{l_1 \cdots l_n}_{m_1 (m_2=1) \cdots m_n},\cdots, \\
&O^{l_1 \cdots l_n}_{m_1 \cdots (m_n=3)}\cap O^{l_1 \cdots l_n}_{m_1 \cdots (m_n=1)},
\end{align*}
and otherwise are trivial. We define the transition function on $O^{l_1 \cdots (l_j =3) \cdots l_n}_{m_1 \cdots m_n}\cap O^{l_1 \cdots (l_j =1) \cdots l_n}_{m_1 \cdots m_n}$ by
\begin{align*}
&\left.\psi ^{l_1 \cdots (l_j =3) \cdots l_n}_{m_1 \cdots m_n} \right|_{O^{l_1 \cdots (l_j =3) \cdots l_n}_{m_1 \cdots m_n}\cap O^{l_1 \cdots (l_j =1) \cdots l_n}_{m_1 \cdots m_n}}\\
&=e^{\frac{\mathbf{i}}{r}a_j y}V_j \left.\psi ^{l_1 \cdots (l_j =1) \cdots l_n}_{m_1 \cdots m_n} \right|_{O^{l_1 \cdots (l_j =3) \cdots l_n}_{m_1 \cdots m_n}\cap O^{l_1 \cdots (l_j =1) \cdots l_n}_{m_1 \cdots m_n}},
\end{align*}
where $a_j:=(a_{1j},\cdots,a_{nj})\in \mathbb{Z}^n$ and $V_j\in U(r')$. Similarly, we define the transition function on $O^{l_1 \cdots l_n}_{m_1 \cdots (m_k =3) \cdots m_n}\cap O^{l_1 \cdots l_n}_{m_1 \cdots (m_k=1) \cdots m_n}$ by
\begin{align*}
&\left.\psi ^{l_1 \cdots l_n}_{m_1 \cdots (m_k=3) \cdots m_n} \right|_{O^{l_1 \cdots l_n}_{m_1 \cdots (m_k =3) \cdots m_n}\cap O^{l_1 \cdots l_n}_{m_1 \cdots (m_k=1) \cdots m_n}}\\
&=U_k \left.\psi ^{l_1 \cdots l_n}_{m_1 \cdots (m_k=1) \cdots m_n} \right|_{O^{l_1 \cdots l_n}_{m_1 \cdots (m_k =3) \cdots m_n}\cap O^{l_1 \cdots l_n}_{m_1 \cdots (m_k=1) \cdots m_n}},
\end{align*}
where $U_k\in U(r')$. In the definitions of these transition functions, actually, we only treat $V_j$, $U_k\in U(r')$ which satisfy the cocycle condition, so we explain the cocycle condition below. When we define 
\begin{align*}
&\left.\psi ^{l_1 \cdots (l_j =3) \cdots l_n}_{m_1 \cdots (m_k=3) \cdots m_n} \right|_{O^{l_1 \cdots (l_j =3) \cdots l_n}_{m_1 \cdots (m_k =3) \cdots m_n}\cap O^{l_1 \cdots (l_j =1) \cdots l_n}_{m_1 \cdots (m_k=1) \cdots m_n}}\\
&=U_k \left.\psi ^{l_1 \cdots (l_j =3) \cdots l_n}_{m_1 \cdots (m_k=1) \cdots m_n} \right|_{O^{l_1 \cdots (l_j =3) \cdots l_n}_{m_1 \cdots (m_k =3) \cdots m_n}\cap O^{l_1 \cdots (l_j =1) \cdots l_n}_{m_1 \cdots (m_k=1) \cdots m_n}}\\
&=\Bigl (U_k\Bigr )\left (e^{\frac{\mathbf{i}}{r}a_j y}V_j\right ) \left.\psi ^{l_1 \cdots (l_j =1) \cdots l_n}_{m_1 \cdots (m_k=1) \cdots m_n} \right|_{O^{l_1 \cdots (l_j =3) \cdots l_n}_{m_1 \cdots (m_k =3) \cdots m_n}\cap O^{l_1 \cdots (l_j =1) \cdots l_n}_{m_1 \cdots (m_k=1) \cdots m_n}},
\end{align*}
the cocycle condition is expressed as
\begin{equation*}
V_j V_k=V_k V_j,\ U_j U_k=U_k U_j,\ \zeta  ^{-a_{kj}}U_k V_j=V_j U_k,
\end{equation*}
where $\zeta $ is the $r$-th root of 1 and $j,k=1,\cdots,n$. We define a set $\mathcal{U}$ of unitary matrices by 
\begin{align*}
&\mathcal{U}:= \Bigl \{ V_j , U_k \in U(r') \ | \ V_j V_k=V_k V_j,\ U_j U_k=U_k U_j,\ \zeta ^{-a_{kj}}U_k V_j=V_j U_k, \notag \\
& \hspace{42mm} j,k=1,\cdots ,n \Bigr \}. 
\end{align*}
Of course, how to define the set $\mathcal{U}$ relates closely to (in)decomposability of $E_{(r,A,r',\mathcal{U},p,q)}$. Here, we only treat the set $\mathcal{U}$ such that $E_{(r,A,r',\mathcal{U},p,q)}$ is simple. Actually, we can take such a set $\mathcal{U}\not=\emptyset $ for any $(r,A,r')\in \mathbb{N}\times M(n;\mathbb{Z})\times \mathbb{N}$, and this fact is discussed in Proposition \ref{simplicity}. Furthermore, we define a connection $\nabla_{(r,A,r',\mathcal{U},p,q)}$ on $E_{(r,A,r',\mathcal{U},p,q)}$ locally as 
\begin{equation*}
\nabla_{(r,A,r',\mathcal{U},p,q)}:=d-\frac{\mathbf{i}}{2\pi } \left( \left( \frac{1}{r}x^t A^t +\frac{1}{r}p^t \right) +\frac{1}{r}q^t T \right) dy\cdot I_{r'},
\end{equation*}
where $dy:=(dy_1,\cdots,dy_n)^t$ and $d$ denotes the exterior derivative. In fact, $\nabla_{(r,A,r',\mathcal{U},p,q)}$ is compatible with the transition functions and so defines a global connection. Then, its curvature form $\Omega _{(r,A,r',\mathcal{U},p,q)}$ is expressed locally as
\begin{equation}
\Omega _{(r,A,r',\mathcal{U},p,q)}=-\frac{\mathbf{i}}{2\pi r}dx^t A^t dy\cdot I_{r'}, \label{curvature}
\end{equation}
where $dx:=(dx_1,\cdots,dx_n)^t$. In particular, this local expression (\ref{curvature}) implies that holomorphic vector bundles $E_{(r,A,r',\mathcal{U},p,q)}$ are simple projectively flat bundles (for example, the definition of projectively flat bundles is written in \cite{koba}). Moreover, the interpretation for these simple projectively flat bundles $E_{(r,A,r',\mathcal{U},p,q)}$ by using factors of automorphy is given in section 3 of \cite{kazushi2}. Now, we consider the condition such that $E_{(r,A,r',\mathcal{U},p,q)}$ is holomorphic. We see that the following proposition holds.
\begin{proposition}\label{prophol}
For a given quadruple $(r,A,p,q)\in \mathbb{N}\times M(n;\mathbb{Z})\times \mathbb{R}^n\times \mathbb{R}^n$, the complex vector bundle $E_{(r,A,r',\mathcal{U},p,q)}\rightarrow T^{2n}_{J=T}$ is holomorphic if and only if $AT=(AT)^t$ holds.
\end{proposition}
\begin{proof}
A complex vector bundle is holomorphic if and only if the (0,2)-part of its curvature form vanishes, so we calculate the (0,2)-part of $\Omega _{(r,A,\mu ,\mathcal{U})}$. It turns out to be
\begin{equation*}
\Omega ^{(0,2)}_{(r,A,r',\mathcal{U},p,q)}=\frac{\mathbf{i}}{2\pi r}d\bar{z}^t\{ T(T-\bar{T} )^{-1} \}^t A^t (T-\bar{T})^{-1}d\bar{z}\cdot I_{r'} ,
\end{equation*}
where $d\bar{z} :=(d\bar{z}_1,\cdots,d\bar{z}_n)^t$. Thus, $\Omega ^{(0,2)}_{(r,A,r',\mathcal{U},p,q)}=0$ is equivalent to that $\{ T(T-\bar{T} )^{-1} \}^t A^t (T-\bar{T} )^{-1}$ is a symmetric matrix, i.e., $AT=(AT)^t$.
\end{proof}
Concerning the above discussions, here, we mention the simplicity of holomorphic vector bundles $E_{(r,A,r',\mathcal{U},p,q)}$. The following proposition holds.
\begin{proposition} \label{simplicity}
For each quadruple $(r,A,p,q)\in \mathbb{N}\times M(n;\mathbb{Z})\times \mathbb{R}^n\times \mathbb{R}^n$, we can take a set $\mathcal{U}\not=\emptyset $ such that $E_{(r,A,r',\mathcal{U},p,q)}$ is simple.
\end{proposition}
\begin{proof}
For a given pair $(r,A)\in \mathbb{N}\times M(n;\mathbb{Z})$, we can take two matrices $\mathcal{A}$, $\mathcal{B}\in GL(n;\mathbb{Z})$ which satisfy the relation (\ref{matAB}). Then, note that $r':=r_1'\cdots r_s'\in \mathbb{N}$ is uniquely defined in the sense of the relation (\ref{r'}). We fix such matrices $\mathcal{A}$, $\mathcal{B}\in GL(n;\mathbb{Z})$, and set
\begin{equation*}
T':=\mathcal{B}^{-1} T \mathcal{A}^t.
\end{equation*}
By using this $T'$, we can consider the complex torus $T^{2n}_{J=T'}=\mathbb{C}^n/2\pi (\mathbb{Z}^n\oplus T'\mathbb{Z}^n)$, and we locally express the complex coordinates $Z:=(Z_1,\cdots, Z_n)^t$ of $T^{2n}_{J=T'}$ by $Z=X+T'Y$, where $X:=(X_1,\cdots, X_n)^t$, $Y:=(Y_1,\cdots, Y_n)^t$. In particular, the complex torus $T^{2n}_{J=T'}$ is biholomorphic to the complex torus $T^{2n}_{J=T}$, and the biholomorphic map
\begin{equation*}
\varphi : T^{2n}_{J=T} \stackrel{\sim }{\rightarrow } T^{2n}_{J=T'}
\end{equation*}
is actually given by
\begin{equation*}
\varphi (z)=\mathcal{B}^{-1}z.
\end{equation*}
Furthermore, when we regard the complex manifolds $T^{2n}_{J=T}$ and $T^{2n}_{J=T'}$ as the real differentiable manifolds $\mathbb{R}^{2n}/2\pi \mathbb{Z}^{2n}$, the biholomorphic map $\varphi $ is regarded as the diffeomorphism
\begin{equation*}
\left( \begin{array}{ccc} X \\ Y \end{array} \right)=\varphi \left( \begin{array}{ccc} x \\ y \end{array} \right) = \left( \begin{array}{ccc} \mathcal{B}^{-1} & O \\ O & (\mathcal{A}^{-1})^t \end{array} \right) \left( \begin{array}{ccc} x \\ y \end{array} \right).
\end{equation*}

Now, we define a set $\mathcal{U}'\not=\emptyset $ as follows. First, we set
\begin{equation*}
\tilde{V}_i:=\left( \begin{array}{cccc} 0\ \ 1&&\\&\ddots\ddots&\\&&1\\1&&0 \end{array} \right)\in U(r_i'), \ \tilde{U}_i:=\left( \begin{array}{cccc} 1&&&\\&\zeta _i&&\\&&\ddots&\\&&&\zeta _i ^{r_i'-1} \end{array} \right)\in U(r_i'), 
\end{equation*}
where $\zeta _i:=e^{\frac{2\pi \mathbf{i}}{r_i'}}$, $i=1,\cdots, s$. By using these matrices $\tilde{V}_i$, $\tilde{U}_i$, we define
\begin{align*}
&V_1':=\tilde{V}_1 \otimes I_{r_2'} \otimes \cdots \otimes I_{r_s'} \in U(r'), \\
&V_2':=I_{r_1'} \otimes \tilde{V}_2 \otimes \cdots \otimes I_{r_s'} \in U(r'), \\
& \hspace{30mm} \vdots \\
&V_s':=I_{r_1'} \otimes \cdots \otimes I_{r_{s-1}'} \otimes \tilde{V}_s \in U(r'), \\
&V_{s+1}':=I_{r_1'} \otimes \cdots \otimes I_{r_s'} \in U(r'), \\
& \hspace{30mm} \vdots \\
&V_n':=I_{r_1'} \otimes \cdots \otimes I_{r_s'} \in U(r'), \\
&U_1':=\tilde{U}_1^{-a_1'} \otimes I_{r_2'} \otimes \cdots \otimes I_{r_s'} \in U(r'), \\
&U_2':=I_{r_1'} \otimes \tilde{U}_2^{-a_2'} \otimes \cdots \otimes I_{r_s'} \in U(r'), \\
& \hspace{30mm} \vdots \\
&U_s':=I_{r_1'} \otimes \cdots \otimes I_{r_{s-1}'} \otimes \tilde{U}_s^{-a_s'} \in U(r'), \\
&U_{s+1}':=I_{r_1'} \otimes \cdots \otimes I_{r_s'} \in U(r'), \\
& \hspace{30mm} \vdots \\
&U_n':=I_{r_1'} \otimes \cdots \otimes I_{r_s'} \in U(r'),
\end{align*}
and set
\begin{equation*}
\mathcal{U}':= \Bigl\{ V_1', V_2', \cdots, V_s', V_{s+1}', \cdots, V_n', U_1', U_2', \cdots, U_s', U_{s+1}', \cdots, U_n' \in U(r') \Bigr\} \not=\emptyset .
\end{equation*}
Then, we can construct the holomorphic vector bundle
\begin{equation*}
E_{(r,\mathcal{A}A\mathcal{B},r',\mathcal{U}',\mathcal{A}p,\mathcal{B}^tq)}\rightarrow T^{2n}_{J=T'},
\end{equation*}
and in particular, $V_j'$ and $U_k'$ are used in the definition of the transition functions of $E_{(r,\mathcal{A}A\mathcal{B},r',\mathcal{U}',\mathcal{A}p,\mathcal{B}^tq)}$ in the $X_j$ and $Y_k$ directions, respectively ($j$, $k=1,\cdots, n$). For this holomorphic vector bundle $E_{(r,\mathcal{A}A\mathcal{B},r',\mathcal{U}',\mathcal{A}p,\mathcal{B}^tq)}$, we can consider the pullback bundle $\varphi ^*E_{(r,\mathcal{A}A\mathcal{B},r',\mathcal{U}',\mathcal{A}p,\mathcal{B}^tq)}$ by the biholomorphic map $\varphi $, and we can regard the pullback bundle $\varphi ^*E_{(r,\mathcal{A}A\mathcal{B},r',\mathcal{U}',\mathcal{A}p,\mathcal{B}^tq)}$ as the holomorphic vector bundle
\begin{equation*}
E_{(r,A,r',\mathcal{U},p,q)}\rightarrow T^{2n}_{J=T}
\end{equation*}
by using a suitable set $\mathcal{U}$ which is defined by employing the data $(\mathcal{U}',\mathcal{A},\mathcal{B})$. In particular, since we can check $\mathcal{U}\not=\emptyset $ easily, in order to prove the statement of this proposition, it is enough to prove that $E_{(r,\mathcal{A}A\mathcal{B},r',\mathcal{U}',\mathcal{A}p,\mathcal{B}^tq)}$ is simple.

Let
\begin{equation*}
\varphi ^{r_1'\cdots r_s'}
\end{equation*}
be an element in $\mathrm{End}(E_{(r,\mathcal{A}A\mathcal{B},r',\mathcal{U}',\mathcal{A}p,\mathcal{B}^tq)})$. Since the rank of $E_{(r,\mathcal{A}A\mathcal{B},r',\mathcal{U}',\mathcal{A}p,\mathcal{B}^tq)}$ is $r'=r_1'\cdots r_s'$, we can treat $\varphi ^{r_1'\cdots r_s'}$ as a matrix of order $r'$. Then, we can divide $\varphi ^{r_1'\cdots r_s'}$ as follows. 
\begin{equation*}
\varphi ^{r_1'\cdots r_s'}=\left( \begin{array}{cccc} \varphi _{11}^{r_2'\cdots r_s'} & \cdots & \varphi _{1r_1'}^{r_2'\cdots r_s'} \\ \vdots & \ddots & \vdots \\ \varphi _{r_1'1}^{r_2'\cdots r_s'} & \cdots & \varphi _{r_1'r_1'}^{r_2'\cdots r_s'} \end{array} \right) = \left( \varphi _{k_1l_1}^{r_2'\cdots r_s'} \right)_{1\leq k_1, l_1 \leq r_1'}.
\end{equation*}
Here, each $\varphi _{k_1l_1}^{r_2'\cdots r_s'}$ is a matrix of order $r_2'\cdots r_s'$. Similarly, we can divide each $\varphi _{k_1l_1}^{r_2'\cdots r_s'}$ as follows.
\begin{equation*}
\varphi _{k_1l_1}^{r_2'\cdots r_s'}=\left( \begin{array}{cccc} \left( \varphi _{k_1l_1}^{r_2'\cdots r_s'} \right)_{11}^{r_3'\cdots r_s'} & \cdots & \left( \varphi _{k_1l_1}^{r_2'\cdots r_s'} \right)_{1r_2'}^{r_3'\cdots r_s'} \\ \vdots & \ddots & \vdots \\ \left( \varphi _{k_1l_1}^{r_2'\cdots r_s'} \right)_{r_2'1}^{r_3'\cdots r_s'} & \cdots & \left( \varphi _{k_1l_1}^{r_2'\cdots r_s'} \right)_{r_2'r_2'}^{r_3'\cdots r_s'} \end{array} \right) = \left( \left( \varphi _{k_1l_1}^{r_2'\cdots r_s'} \right)_{k_2l_2}^{r_3'\cdots r_s'} \right)_{1\leq k_2, l_2 \leq r_2'}.
\end{equation*}
Here, each $\left( \varphi _{k_1l_1}^{r_2'\cdots r_s'} \right)_{k_2l_2}^{r_3'\cdots r_s'}$ is matrix of order $r_3'\cdots r_s'$. By repeating the above steps, as a result, we can express $\varphi ^{r_1'\cdots r_s'}$ as
\begin{equation*}
\left( \left( \left( \cdots \left( \left( \varphi _{k_1l_1}^{r_2'\cdots r_s'} \right)_{k_2l_2}^{r_3'\cdots r_s'} \right) \cdots \right)_{k_{s-1}l_{s-1}}^{r_s'} \right)_{k_sl_s} \right),
\end{equation*}
where $1\leq k_i, l_i\leq r_i'$ ($i=1,\cdots, s$). Hereafter, we consider the local expression of each component
\begin{equation}
\left( \left( \cdots \left( \left( \varphi _{k_1l_1}^{r_2'\cdots r_s'} \right)_{k_2l_2}^{r_3'\cdots r_s'} \right) \cdots \right)_{k_{s-1}l_{s-1}}^{r_s'} \right)_{k_sl_s} \label{component}
\end{equation}
of $\varphi ^{r_1'\cdots r_s'}$. First, for each $k=1,\cdots, n$, by considering the transition functions of $E_{(r,\mathcal{A}A\mathcal{B},r',\mathcal{U}',\mathcal{A}p,\mathcal{B}^tq)}$ in the $Y_k$ direction, we see that the morphism $\varphi ^{r_1'\cdots r_s'}$ must satisfy 
\begin{align}
&U_k'\cdot \varphi ^{r_1'\cdots, r_s'}(X_1,\cdots, X_n, Y_1,\cdots, Y_n) \notag \\
&= \varphi ^{r_1'\cdots r_s'} (X_1,\cdots, X_n, Y_1,\cdots, Y_k+2\pi ,\cdots, Y_n) \cdot U_k'. \label{Y}
\end{align}
Furthermore, the morphism $\varphi ^{r_1'\cdots r_s'}$ need to satisfy not only the relation (\ref{Y}) but also the Cauchy-Riemann equation
\begin{equation}
\bar{\partial } (\varphi ^{r_1'\cdots r_s'})=O. \label{C-R}
\end{equation}
Therefore, by the relations (\ref{Y}) and (\ref{C-R}), we can give a local expression of the component (\ref{component}) of $\varphi ^{r_1'\cdots, r_s'}$ as follows.
\begin{align*}
\sum_{I_{k_1l_1}^1,\cdots, I_{k_sl_s}^s, I^{s+1},\cdots, I^n \in \mathbb{Z}} & \lambda _{I_{k_1l_1}^1,\cdots, I_{k_sl_s}^s, I^{s+1},\cdots, I^n}^{k_1l_1,\cdots, k_sl_s} \\
&e^{\mathbf{i}\left( \frac{a_1'}{r_1'}(l_1-k_1)+I_{k_1l_1}^1,\cdots, \frac{a_s'}{r_s'}(l_s-k_s)+I_{k_sl_s}^s, I^{s+1},\cdots, I^n \right) T^{-1}(X+TY)}.
\end{align*}
Here, 
\begin{equation*}
\lambda _{I_{k_1l_1}^1,\cdots, I_{k_sl_s}^s, I^{s+1},\cdots, I^n}^{k_1l_1,\cdots, k_sl_s} \in \mathbb{C}
\end{equation*}
is an arbitrary constant. Finally, for each $j=1,\cdots, n$, we consider the conditions on the transition functions of $E_{(r,\mathcal{A}A\mathcal{B},r',\mathcal{U}',\mathcal{A}p,\mathcal{B}^tq)}$ in the $X_j$ direction. By a direct calculation, if $j=1,\cdots, s$, we obtain 
\begin{align}
&\left( \cdots \left( \varphi _{k_1l_1}^{r_2'\cdots r_s'} \right) \cdots \right)_{k_jl_j}^{r_{j+1}'\cdots r_s'}(X_1,\cdots, X_j+2\pi ,\cdots, X_n, Y_1,\cdots, Y_n) \notag \\
&=\left( \cdots \left( \varphi _{k_1l_1}^{r_2'\cdots r_s'} \right) \cdots \right)_{(k_j+1)(l_j+1)}^{r_{j+1}'\cdots r_s'}(X_1,\cdots, X_n, Y_1,\cdots, Y_n), \label{X1}
\end{align}
and if $j=s+1,\cdots, n$, we obtain 
\begin{align}
&\left( \cdots \left( \varphi _{k_1l_1}^{r_2'\cdots r_s'} \right) \cdots \right)_{k_sl_s}(X_1,\cdots, X_j+2\pi ,\cdots, X_n,Y_1,\cdots, Y_n) \notag \\
&=\left( \cdots \left( \varphi _{k_1l_1}^{r_2'\cdots r_s'} \right) \cdots \right)_{k_sl_s}(X_1,\cdots, X_n,Y_1,\cdots, Y_n). \label{X2}
\end{align}
In particular, the relation (\ref{X1}) implies 
\begin{align}
&\left( \cdots \left( \varphi _{k_1l_1}^{r_2'\cdots r_s'} \right) \cdots \right)_{k_sl_s}(X_1,\cdots, X_j+2\pi r_j',\cdots, X_n,Y_1,\cdots, Y_n) \notag \\
&=\left( \cdots \left( \varphi _{k_1l_1}^{r_2'\cdots r_s'} \right) \cdots \right)_{k_sl_s}(X_1,\cdots, X_n,Y_1,\cdots, Y_n), \label{X3}
\end{align}
and by using the relations (\ref{X2}) and (\ref{X3}), we have the condition
\begin{align}
&\left( \frac{a_1'}{r_1'}(l_1-k_1)+I_{k_1l_1}^1,\cdots, \frac{a_s'}{r_s'}(l_s-k_s)+I_{k_sl_s}^s,I^{s+1},\cdots, I^n \right) T'^{-1} \notag \\
&\in \frac{\mathbb{Z}}{r_1'}\times \cdots \frac{\mathbb{Z}}{r_s'}\times \mathbb{Z}\times \cdots \mathbb{Z} \subset \mathbb{R}^n. \label{X4}
\end{align}
Now, since $\mathrm{Im}T'^{-1}$ is positive definite, the condition (\ref{X4}) turns out to be
\begin{equation}
\left( \frac{a_1'}{r_1'}(l_1-k_1)+I_{k_1l_1}^1,\cdots, \frac{a_s'}{r_s'}(l_s-k_s)+I_{k_sl_s}^s,I^{s+1},\cdots, I^n \right) =0. \label{X5}
\end{equation}
Here, recall the relation (\ref{X1}). By the relation (\ref{X1}), it is enough to consider the condition (\ref{X5}) in the case $k_1=\cdots =k_s=1$. We focus on the first component in the condition (\ref{X5}), i.e.,
\begin{equation}
\frac{a_1'}{r_1'}(l_1-1)+I_{1l_1}^1=0. \label{X6}
\end{equation}
In the case $l_1=1$, the condition (\ref{X6}) turns out to be $I_{11}^1=0$, so by the relation (\ref{X1}), we see
\begin{equation*}
\varphi _{11}^{r_2'\cdots r_s'}=\varphi _{22}^{r_2'\cdots r_s'}=\cdots =\varphi _{r_1'r_1'}^{r_2'\cdots r_s'} \in M(r_2'\cdots r_s';\mathbb{C}).
\end{equation*}
We consider the cases $l_1=2,\cdots, r_1'$. Note that $gcd(r_1',a_1')=1$ holds by the assumption. Therefore, we have
\begin{equation*}
l_1-1\in r_1'\mathbb{Z}
\end{equation*}
by the condition (\ref{X6}). However, this fact contradicts the assumption $l_1=2,\cdots, r_1'$. Thus, for each $l_1=2,\cdots, r_1'$, we obtain
\begin{equation*}
\varphi _{1l_1}^{r_2'\cdots r_s'}=O,
\end{equation*}
and by using the relation (\ref{X1}) again, we also obtain
\begin{equation*}
\varphi _{2(l_1+1)}^{r_2'\cdots r_s'}=\varphi _{3(l_1+2)}^{r_2'\cdots r_s'}=\cdots \varphi _{r_1'(l_1-1)}^{r_2'\cdots r_s'}=O.
\end{equation*}
Similarly, by focusing on the second component in the condition (\ref{X5}), we see
\begin{equation*}
\left( \varphi _{11}^{r_2'\cdots r_s'} \right)_{11}^{r_3'\cdots r_s'}=\left( \varphi _{11}^{r_2'\cdots r_s'} \right)_{22}^{r_3'\cdots r_s'}=\cdots =\left( \varphi _{11}^{r_2'\cdots r_s'} \right)_{r_2'r_2'}^{r_3'\cdots r_s'} \in M(r_3'\cdots r_s';\mathbb{C}),
\end{equation*}
and the other components $\left( \varphi _{11}^{r_2'\cdots r_s'} \right)_{k_2l_2}^{r_3'\cdots r_s'}$ of the matrix $\varphi _{11}^{r_2'\cdots r_s'}$ vanish. By repeating the above discussions, as a result, we have
\begin{equation*}
\varphi ^{r_1'\cdots r_s'}=\left( \left( \cdots \left( \left( \varphi _{11}^{r_2'\cdots r_s'} \right)_{11}^{r_3'\cdots r_s'} \right) \cdots \right)_{11}^{r_s'} \right)_{11} \cdot I_{r'},
\end{equation*}
where
\begin{equation*}
\left( \left( \cdots \left( \left( \varphi _{11}^{r_2'\cdots r_s'} \right)_{11}^{r_3'\cdots r_s'} \right) \cdots \right)_{11}^{r_s'} \right)_{11} \in \mathbb{C}.
\end{equation*}
Thus, we can conclude
\begin{equation*}
\mathrm{End}(E_{(r,\mathcal{A}A\mathcal{B},r',\mathcal{U}',\mathcal{A}p,\mathcal{B}^tq)})\cong \mathbb{C}.
\end{equation*}
\end{proof}

We define a DG-category 
\begin{equation*}
DG_{T^{2n}_{J=T}}
\end{equation*}
consisting of holomorphic vector bundles $(E_{(r,A,r',\mathcal{U},p,q)},\nabla_{(r,A,r',\mathcal{U},p,q)})$. This definition is an extension of the case of a pair $(T^{2}_{J=T},\check{T}^{2}_{J=T})$ of elliptic curves in \cite{kazushi} (see section 3) to the higher dimensional case, and it is also written in section 2 of \cite{kazushi2}. The objects of $DG_{T^{2n}_{J=T}}$ are holomorphic vector bundles $E_{(r,A,r',\mathcal{U},p,q)}$ with $U(r')$-connections $\nabla_{(r,A,r',\mathcal{U},p,q)}$. Of course, we assume $AT=(AT)^t$. Sometimes we simply denote $(E_{(r,A,\mu ,\mathcal{U})},\nabla_{(r,A,\mu ,\mathcal{U})})$ by $E_{(r,A,r',\mathcal{U},p,q)}$. For any two objects 
\begin{align*}
&E_{(r,A,r',\mathcal{U},p,q)}=(E_{(r,A,r',\mathcal{U},p,q)},\nabla_{(r,A,r',\mathcal{U},p,q)}), \\
&E_{(s,B,s',\mathcal{V},u,v)}=(E_{(s,B,s',\mathcal{V},u,v)},\nabla_{(s,B,s',\mathcal{V},u,v)}), 
\end{align*}
the space of morphisms is defined by
\begin{align*}
&\mathrm{Hom}_{DG_{T^{2n}_{J=T}}}(E_{(r,A,r',\mathcal{U},p,q)},E_{(s,B,s',\mathcal{V},u,v)}) \\
&:=\Gamma (E_{(r,A,r',\mathcal{U},p,q)},E_{(s,B,s',\mathcal{V},u,v)})\bigotimes _{C^{\infty }(T^{2n}_{J=T})}\Omega ^{0,*}(T^{2n}_{J=T}),
\end{align*}
where $\Omega ^{0,*}(T^{2n}_{J=T})$ is the space of anti-holomorphic differential forms, and 
\begin{equation*}
\Gamma (E_{(r,A,r',\mathcal{U},p,q)},E_{(s,B,s',\mathcal{V},u,v)})
\end{equation*}
is the space of homomorphisms from $E_{(r,A,r',\mathcal{U},p,q)}$ to $E_{(s,B,s',\mathcal{V},u,v)}$. The space of morphisms $\mathrm{Hom}_{DG_{T^{2n}_{J=T}}}(E_{(r,A,r',\mathcal{U},p,q)},E_{(s,B,s',\mathcal{V},u,v)})$ is a $\mathbb{Z}$-graded vector space, where the grading is defined as the degree of the anti-holomorphic differential forms. In particular, the degree $r$ part is denoted 
\begin{equation*}
\mathrm{Hom}^r_{DG_{T^{2n}_{J=T}}}(E_{(r,A,r',\mathcal{U},p,q)},E_{(s,B,s',\mathcal{V},u,v)}). 
\end{equation*}
We decompose $\nabla_{(r,A,r',\mathcal{U},p,q)}$ into its holomorphic part and anti-holomorphic part $\nabla_{(r,A,r',\mathcal{U},p,q)}=\nabla ^{(1,0)}_{(r,A,r',\mathcal{U},p,q)}+\nabla ^{(0,1)}_{(r,A,r',\mathcal{U},p,q)}$, and define a linear map
\begin{equation*}
\mathrm{Hom}^r_{DG_{T^{2n}_{J=T}}}(E_{(r,A,r',\mathcal{U},p,q)},E_{(s,B,s',\mathcal{V},u,v)})\rightarrow \mathrm{Hom}^{r+1}_{DG_{T^{2n}_{J=T}}}(E_{(r,A,r',\mathcal{U},p,q)},E_{(s,B,s',\mathcal{V},u,v)})
\end{equation*}
by
\begin{equation*}
\psi \mapsto ( 2\nabla^{(0,1)}_{(s,B,s',\mathcal{V},u,v)} )(\psi )-(-1)^r\psi  (2\nabla ^{(0,1)}_{(r,A,r',\mathcal{U},p,q)} ).
\end{equation*}
We can check that this linear map is a differential. Furthermore, the product structure is defined by the composition of homomorphisms of vector bundles together with the wedge product for the anti-holomorphic differential forms. Then, these differential and product structure satisfy the Leibniz rule. Thus, $DG_{T^{2n}_{J=T}}$ forms a DG-category.

\subsection{The isomorphism classes of $E_{(r,A,r',\mathcal{U},p,q)}$}
In this subsection, we fix $r\in \mathbb{N}$, $A\in M(n;\mathbb{Z})$, and consider the condition such that $E_{(r,A,r',\mathcal{U},p,q)}\cong E_{(r,A,r',\mathcal{U}',p',q')}$ holds. Here, 
\begin{equation*}
p, q, p', q'\in \mathbb{R}^n
\end{equation*}
and
\begin{align*}
&\mathcal{U}:= \Bigl \{ V_j , U_k \in U(r') \ | \ V_j V_k=V_k V_j,\ U_j U_k=U_k U_j,\ \zeta ^{-a_{kj}}U_k V_j=V_j U_k, \\
& \hspace{42mm} j,k=1,\cdots ,n \Bigr \}, \\
&\mathcal{U}':= \Bigl \{ V'_j , U'_k \in U(r') \ | \ V'_j V'_k=V'_k V'_j,\ U'_j U'_k=U'_k U'_j,\ \zeta ^{-a_{kj}}U'_k V'_j=V'_j U'_k, \\
& \hspace{42mm} j,k=1,\cdots ,n \Bigr \}. 
\end{align*}
Furthermore, for each $j=1,\cdots,n$, we define $\xi _j$, $\theta _j$, $\xi '_j$, $\theta '_j \in \mathbb{R}$ by
\begin{equation*}
e^{\mathbf{i}\xi _j}=\mathrm{det}V_j,\ e^{\mathbf{i}\theta _j}=\mathrm{det}U_j,\ e^{\mathbf{i}\xi '_j}=\mathrm{det}V'_j,\ e^{\mathbf{i}\theta '_j}=\mathrm{det}U'_j,
\end{equation*}
and set
\begin{equation}
\xi :=(\xi _1,\cdots,\xi _n)^t,\ \theta :=(\theta _1,\cdots,\theta _n)^t,\ \xi ':=(\xi '_1,\cdots,\xi '_n)^t,\ \theta ':=(\theta '_1,\cdots,\theta '_n)^t\in \mathbb{R}^n. \label{xitheta}
\end{equation}

Now, in order to consider the condition such that $E_{(r,A,r',\mathcal{U},p,q)}\cong E_{(r,A,r',\mathcal{U}',p',q')}$ holds, we recall the following classification result of simple projectively flat bundles on complex tori by Matsushima and Mukai (see \cite[Theorem 6.1]{matsu}, \cite[Proposition 6.17 (1)]{mukai}, and note that the notion of semi-homogeneous vector bundles in \cite{mukai} is equivalent to the notion of projectively flat bundles).
\begin{theo}[Matsushima \cite{matsu}, Mukai \cite{mukai}] \label{theoremMM}
For two simple projectively flat bundles $E$, $E'$ over a complex torus $\mathbb{C}^n/\Gamma $ $(\Gamma \subset \mathbb{C}^n$ is a lattice$)$ which satisfy $(\mathrm{rank}E, c_1(E))=(\mathrm{rank}E', c_1(E'))$, there exists a line bundle $L\in \mathrm{Pic}^0(\mathbb{C}^n/\Gamma )$ such that
\begin{equation*}
E'\cong E\otimes L.
\end{equation*}
\end{theo}
By using Theorem \ref{theoremMM}, we obtain the following theorem.
\begin{theo} \label{EE'}
Two holomorphic vector bundles $E_{(r,A,r',\mathcal{U},p,q)}$, $E_{(r,A,r',\mathcal{U}',p',q')}$ are isomorphic to each other, 
\begin{equation*}
E_{(r,A,r',\mathcal{U},p,q)}\cong E_{(r,A,r',\mathcal{U}',p',q')},
\end{equation*}
if and only if 
\begin{equation*}
(p+T^t q)-(p'+T^t q') \equiv \frac{r}{r'}(\theta -\theta ')+T^t \frac{r}{r'}(\xi '-\xi ) \ ( \mathrm{mod} \ 2\pi r ( \mathcal{A}^{-1} \left( \begin{array}{cccc} \frac{\mathbb{Z}}{r_1'} \\ \vdots \\ \frac{\mathbb{Z}}{r_s'} \\ \mathbb{Z} \\ \vdots \\ \mathbb{Z} \end{array} \right) \oplus T^t (\mathcal{B}^{-1})^t \left( \begin{array}{cccc} \frac{\mathbb{Z}}{r_1'} \\ \vdots \\ \frac{\mathbb{Z}}{r_s'} \\ \mathbb{Z} \\ \vdots \\ \mathbb{Z} \end{array} \right) ) )
\end{equation*}
holds.
\end{theo}
\begin{proof}
In order to prove the statement of this theorem, we again take the biholomorphic map
\begin{equation*}
\varphi : T^{2n}_{J=T} \stackrel{\sim }{\rightarrow } T^{2n}_{J=T'}
\end{equation*}
in the proof of Proposition \ref{simplicity}. Of course, we can also regard this biholomorphic map $\varphi $ as the diffeomorphism $\mathbb{R}^{2n}/2\pi \mathbb{Z}^{2n} \stackrel{\sim }{\rightarrow } \mathbb{R}^{2n}/2\pi \mathbb{Z}^{2n}$ which is expressed locally as
\begin{equation*}
\varphi \left( \begin{array}{ccc} x \\ y \end{array} \right) = \left( \begin{array}{ccc} \mathcal{B}^{-1} & O \\ O & (\mathcal{A}^{-1})^t \end{array} \right) \left( \begin{array}{ccc} x \\ y \end{array} \right).
\end{equation*}
Then, by the biholomorphicity of the map $\varphi$, 
\begin{equation*}
E_{(r,A,r',\mathcal{U},p,q)}\cong E_{(r,A,r',\mathcal{U}',p',q')}
\end{equation*}
holds if and only if
\begin{equation*}
(\varphi ^{-1})^*E_{(r,A,r',\mathcal{U},p,q)}\cong (\varphi ^{-1})^*E_{(r,A,r',\mathcal{U}',p',q')}
\end{equation*}
holds, so we consider the condition such that $(\varphi ^{-1})^*E_{(r,A,r',\mathcal{U},p,q)}\cong (\varphi ^{-1})^*E_{(r,A,r',\mathcal{U}',p',q')}$ holds. Now, by using the suitable sets $\tilde{\mathcal{U}}$ and $\tilde{\mathcal{U}}'$ (the definitions of $\tilde{\mathcal{U}}$ and $\tilde{\mathcal{U}}'$ depend on the data $(\mathcal{U},\mathcal{A},\mathcal{B})$ and $(\mathcal{U}',\mathcal{A},\mathcal{B})$, respectively), we can regard $(\varphi ^{-1})^*E_{(r,A,r',\mathcal{U},p,q)}$ and $(\varphi ^{-1})^*E_{(r,A,r',\mathcal{U}',p',q')}$ as $E_{(r,\tilde{A},r',\tilde{\mathcal{U}},\tilde{p},\tilde{q})}$ and $E_{(r,\tilde{A},r',\tilde{\mathcal{U}}',\tilde{p}',\tilde{q}')}$, respectively, where
\begin{equation*}
\tilde{A}:=\mathcal{A}A\mathcal{B}=\left( \begin{array}{ccccccc} \tilde{a_1} & & & & & \\ & \ddots & & & & \\ & & \tilde{a_s} & & & \\ & & & 0 & & \\ & & & & \ddots & \\ & & & & & 0 \end{array} \right), \ \tilde{p}:=\mathcal{A}p, \ \tilde{q}:=\mathcal{B}^t q, \ \tilde{p}':=\mathcal{A}p', \ \tilde{q}':=\mathcal{B}^t q'.
\end{equation*}
Then, constant vectors $\xi $, $\theta $, $\xi '$, $\theta '\in \mathbb{R}^n$ are also transformed to constant vectors
\begin{equation*}
\tilde{\xi }:=\mathcal{B}^t \xi , \ \tilde{\theta }:=\mathcal{A}\theta , \ \tilde{\xi }':=\mathcal{B}^t \xi ', \ \tilde{\theta }':=\mathcal{A}\theta ' \in \mathbb{R}^n,
\end{equation*}
respectively. For these holomorphic vector bundles $E_{(r,\tilde{A},r',\tilde{\mathcal{U}},\tilde{p},\tilde{q})}$ and $E_{(r,\tilde{A},r',\tilde{\mathcal{U}}',\tilde{p}',\tilde{q}')}$, by Theorem \ref{theoremMM}, we see that there exists a holomorphic line bundle $E_{(1,O,1,\mathcal{V},u,v)}$ such that
\begin{equation*}
E_{(r,\tilde{A},r',\tilde{\mathcal{U}}',\tilde{p}',\tilde{q}')} \cong E_{(r,\tilde{A},r',\tilde{\mathcal{U}},\tilde{p},\tilde{q})}\otimes E_{(1,O,1,\mathcal{V},u,v)}.
\end{equation*}
Here, 
\begin{equation*}
\mathcal{V}:=\Bigl\{ e^{\mathbf{i}\tau _1},\cdots, e^{\mathbf{i}\tau _n}, e^{\mathbf{i}\sigma _1},\cdots, e^{\mathbf{i}\sigma _n}\in U(1) \Bigr\}, \ u, v\in \mathbb{R}^n,
\end{equation*}
and for simplicity, we set
\begin{equation*}
\tau :=(\tau _1,\cdots, \tau _n)^t, \ \sigma :=(\sigma _1,\cdots, \sigma _n)^t \in \mathbb{R}^n.
\end{equation*}
In particular, for each $j=1,\cdots, n$, we assume that $e^{\mathbf{i}\tau _j}$ and $e^{\mathbf{i}\sigma _j}$ are the transition functions of $E_{(1,O,1,\mathcal{V},u,v)}$ in the $x_j$ direction and the transition function of $E_{(1,O,1,\mathcal{V},u,v)}$ in the $y_j$ direction, respectively. Therefore, since
\begin{equation}
E_{(r,\tilde{A},r',\tilde{\mathcal{U}},\tilde{p},\tilde{q})} \cong E_{(r,\tilde{A},r',\tilde{\mathcal{U}}',\tilde{p}',\tilde{q}')} \label{E}
\end{equation}
holds if and only if 
\begin{equation}
E_{(r,\tilde{A},r',\tilde{\mathcal{U}},\tilde{p},\tilde{q})} \cong E_{(r,\tilde{A},r',\tilde{\mathcal{U}},\tilde{p},\tilde{q})}\otimes E_{(1,O,1,\mathcal{V},u,v)} \label{E1}
\end{equation}
holds, our first goal is to find the relation on the parameters $u$, $v$, $\tau $, $\sigma \in \mathbb{R}^n$ such that the relation (\ref{E1}) holds. By using Theorem \ref{theoremMM} again, we see that there exists a holomorphic line bundle $E_{(1,O,1,\mathcal{V}',u',v')}$ such that
\begin{equation*}
E_{(r,\tilde{A},r',\tilde{\mathcal{U}},\tilde{p},\tilde{q})} \cong E_{(r,\tilde{A},r',\mathcal{U}_0,p_0,q_0)} \otimes E_{(1,O,1,\mathcal{V}',u',v')},
\end{equation*}
where
\begin{equation*}
\mathcal{U}_0:=\Bigl\{ V_1',\cdots, V_n', U_1',\cdots, U_n' \in U(r') \Bigr\}, \ p_0, q_0\in \mathbb{R}^n.
\end{equation*}
Here, note that the definitions of $V_1',\cdots, V_n', U_1',\cdots, U_n' \in U(r')$ are given in the proof of Proposition \ref{simplicity}. Thus, since we can rewrite the relation (\ref{E1}) to
\begin{equation*}
E_{(r,\tilde{A},r',\mathcal{U}_0,p_0,q_0)}\otimes E_{(1,O,1,\mathcal{V}',u',v')} \cong \left( E_{(r,\tilde{A},r',\mathcal{U}_0,p_0,q_0)}\otimes E_{(1,O,1,\mathcal{V}',u',v')} \right) \otimes E_{(1,O,1,\mathcal{V},u,v)},
\end{equation*}
as a result, it is enough to consider the condition such that
\begin{equation*}
E_{(r,\tilde{A},r',\mathcal{U}_0,p_0,q_0)}\cong E_{(r,\tilde{A},r',\mathcal{U}_0,p_0,q_0)}\otimes E_{(1,O,1,\mathcal{V},u,v)}
\end{equation*}
holds. By a direct calculation, we can actually check that 
\begin{equation*}
E_{(r,\tilde{A},r',\mathcal{U}_0,p_0,q_0)}\cong E_{(r,\tilde{A},r',\mathcal{U}_0,p_0,q_0)}\otimes E_{(1,O,1,\mathcal{V},u,v)}
\end{equation*}
holds if and only if 
\begin{equation}
u+T'^t v\equiv \sigma -T'^t \tau \ (\mathrm{mod} \ 2\pi ( \left( \begin{array}{cccc} \frac{\mathbb{Z}}{r_1'} \\ \vdots \\ \frac{\mathbb{Z}}{r_s'} \\ \mathbb{Z} \\ \vdots \\ \mathbb{Z} \end{array} \right) \oplus T'^t \left( \begin{array}{cccc} \frac{\mathbb{Z}}{r_1'} \\ \vdots \\ \frac{\mathbb{Z}}{r_s'} \\ \mathbb{Z} \\ \vdots \\ \mathbb{Z} \end{array} \right) )) \label{nu}
\end{equation}
holds, and in particular, we can regard the relation (\ref{nu}) as 
\begin{align}
&(\tilde{p}+T'^t \tilde{q})-\{ (\tilde{p}+T'^t \tilde{q})+r(u+T'^t v) \}\equiv \frac{r}{r'}\{ \tilde{\theta }-(\tilde{\theta }+r'\sigma ) \} +T'^t \frac{r}{r'} \{ (\tilde{\xi }+r'\tau )-\tilde{\xi } \} \notag \\
&(\mathrm{mod} \ 2\pi r( \left( \begin{array}{cccc} \frac{\mathbb{Z}}{r_1'} \\ \vdots \\ \frac{\mathbb{Z}}{r_s'} \\ \mathbb{Z} \\ \vdots \\ \mathbb{Z} \end{array} \right) \oplus T'^t \left( \begin{array}{cccc} \frac{\mathbb{Z}}{r_1'} \\ \vdots \\ \frac{\mathbb{Z}}{r_s'} \\ \mathbb{Z} \\ \vdots \\ \mathbb{Z} \end{array} \right) )). \label{nu'}
\end{align}
Hence, we see that the relation (\ref{E1}) holds if and only if the relation (\ref{nu'}) holds. Here, note that there exists an isomorphism
\begin{equation*}
\mathrm{det}E_{(r,\tilde{A},r',\tilde{\mathcal{U}},\tilde{p},\tilde{q})}\cong \mathrm{det}(E_{(r,\tilde{A},r',\tilde{\mathcal{U}},\tilde{p},\tilde{q})}\otimes E_{(1,O,1,\mathcal{V},u,v)})
\end{equation*}
with
\begin{align*}
&(\tilde{p}+T'^t \tilde{q})-\{ (\tilde{p}+T'^t \tilde{q})+r(u+T'^t v) \}\equiv \frac{r}{r'}\{ \tilde{\theta }-(\tilde{\theta }+r'\sigma ) \} +T'^t \frac{r}{r'} \{ (\tilde{\xi }+r'\tau )-\tilde{\xi } \} \\
&\left( \mathrm{mod} \ \frac{r}{r'}2\pi (\mathbb{Z}^n\oplus T'^t \mathbb{Z}^n) \right).
\end{align*}
Now, we consider the condition such that the relation (\ref{E}) holds. One necessary condition for the relation (\ref{E}) is that
\begin{equation}
\mathrm{det}E_{(r,\tilde{A},r',\tilde{\mathcal{U}},\tilde{p},\tilde{q})}\cong \mathrm{det}E_{(r,\tilde{A},r',\tilde{\mathcal{U}}',\tilde{p}',\tilde{q}')} \label{det}
\end{equation}
holds, and by a direct calculation, we can rewrite the relation (\ref{det}) to the following :
\begin{equation}
(\tilde{p}+T'^t\tilde{q})-(\tilde{p}'+T'^t\tilde{q}')\equiv \frac{r}{r'}(\tilde{\theta }-\tilde{\theta }')+T'^t \frac{r}{r'}(\tilde{\xi }'-\tilde{\xi }) \ \left( \mathrm{mod} \ \frac{r}{r'}2\pi (\mathbb{Z}^n\oplus T'^t \mathbb{Z}^n) \right). \label{det'}
\end{equation}
Therefore, since the definition of $E_{(1,O,1,\mathcal{V},u,v)}$ is given by
\begin{equation*}
E_{(r,\tilde{A},r',\tilde{\mathcal{U}}',\tilde{p}',\tilde{q}')}\cong E_{(r,\tilde{A},r',\tilde{\mathcal{U}},\tilde{p},\tilde{q})}\otimes E_{(1,O,1,\mathcal{V},u,v)},
\end{equation*}
and the relations (\ref{E1}) and (\ref{nu'}) are equivalent, by considering the relation (\ref{det'}), we can give the condition such that the relation (\ref{E}) holds as follows.
\begin{equation*}
(\tilde{p}+T'^t \tilde{q})-(\tilde{p}'+T'^t \tilde{q}')\equiv \frac{r}{r'}(\tilde{\theta }-\tilde{\theta }')+T'^t \frac{r}{r'}(\tilde{\xi }'-\tilde{\xi }) \ (\mathrm{mod} \ 2\pi r( \left( \begin{array}{cccc} \frac{\mathbb{Z}}{r_1'} \\ \vdots \\ \frac{\mathbb{Z}}{r_s'} \\ \mathbb{Z} \\ \vdots \\ \mathbb{Z} \end{array} \right) \oplus T'^t \left( \begin{array}{cccc} \frac{\mathbb{Z}}{r_1'} \\ \vdots \\ \frac{\mathbb{Z}}{r_s'} \\ \mathbb{Z} \\ \vdots \\ \mathbb{Z} \end{array} \right) )).
\end{equation*}
Thus, by considering the pullback bundles $\varphi ^*E_{(r,\tilde{A},r',\tilde{\mathcal{U}},\tilde{p},\tilde{q})}$ and $\varphi ^*E_{(r,\tilde{A},r',\tilde{\mathcal{U}}',\tilde{p}',\tilde{q}')}$, we can conclude that
\begin{equation*}
E_{(r,A,r',\mathcal{U},p,q)}\cong E_{(r,A,r',\mathcal{U}',p',q')}
\end{equation*}
holds if and only if
\begin{equation*}
(p+T^t q)-(p'+T^t q')\equiv \frac{r}{r'}(\theta -\theta ')+T^t \frac{r}{r'}(\xi '-\xi ) \ (\mathrm{mod} \ 2\pi r( \mathcal{A}^{-1} \left( \begin{array}{cccc} \frac{\mathbb{Z}}{r_1'} \\ \vdots \\ \frac{\mathbb{Z}}{r_s'} \\ \mathbb{Z} \\ \vdots \\ \mathbb{Z} \end{array} \right) \oplus T'^t (\mathcal{B}^{-1})^t \left( \begin{array}{cccc} \frac{\mathbb{Z}}{r_1'} \\ \vdots \\ \frac{\mathbb{Z}}{r_s'} \\ \mathbb{Z} \\ \vdots \\ \mathbb{Z} \end{array} \right) ))
\end{equation*}
holds.
\end{proof}
\begin{rem}
When we work over a pair $(T_{J=T}^2, \check{T}_{J=T}^2)$ of elliptic curves, Theorem \ref{EE'} implies that there exists a one-to-one correspondence between the set of the isomorphism classes of holomorphic vector bundles $E_{(r,A,r',\mathcal{U},p,q)}$ and the set of the isomorphism classes of holomorphic line bundles $\mathrm{det}E_{(r,A,r',\mathcal{U},p,q)}$.
\end{rem}

\section{Symplectic geometry side}
In this section, we define the objects of the Fukaya category corresponding to holomorphic vector bundles $E_{(r,A,r',\mathcal{U},p,q)}\rightarrow T^{2n}_{J=T}$, and study the isomorphism classes of them. The discussions in this section are based on the SYZ construction (SYZ transform) \cite{SYZ} (see also \cite{leung}, \cite{A-P}). 

\subsection{The definition of $\mathscr{L}_{(r,A,p,q)}$}
In this subsection, we define a class of pairs of affine Lagrangian submanifolds 
\begin{equation*}
L_{(r,A,p)}
\end{equation*}
in $\check{T}^{2n}_{J=T}$ and unitary local systems
\begin{equation*}
\mathcal{L}_{(r,A,p,q)}\rightarrow L_{(r,A,p)}.
\end{equation*}

First, we recall the definition of objects of the Fukaya categories following \cite[Definition 1.1]{Fuk}. Let $(M, \Omega)$ be a symplectic manifold $(M, \omega)$ together with a closed 2-form $B$ on $M$. Here, we put $\Omega=\omega+\sqrt{-1}B$ (note $-B+\sqrt{-1}\omega$ is used in many of the literatures). Then, we consider pairs $(L, \mathcal{L})$ with the following properties :
\begin{align}
&L \ \mathrm{is} \ \mathrm{a} \ \mathrm{Lagrangian} \ \mathrm{submanifold} \ \mathrm{of} \ (M,\omega). \label{f1} \\
&\mathcal{L}\rightarrow L \ \mathrm{is} \ \mathrm{a} \ \mathrm{line} \ \mathrm{bundle} \ \mathrm{together} \ \mathrm{with} \ \mathrm{a} \ \mathrm{connection} \ \nabla^{\mathcal{L}} \ \mathrm{such} \ \mathrm{that} \label{f2} \\
&F_{\nabla^{\mathcal{L}}}=2\pi\sqrt{-1}B|_L. \notag
\end{align}
In this context, $F_{\nabla^{\mathcal{L}}}$ denotes the curvature form of the connection $\nabla^{\mathcal{L}}$. We define objects of the Fukaya category on $(M, \Omega)$ by pairs $(L,\mathcal{L})$ which satisfy the properties (\ref{f1}), (\ref{f2}). 

Let us consider the following $n$-dimensional submanifold $\tilde{L}_{(r,A,p)}$ in $\mathbb{R}^{2n}$ :
\begin{equation*}
\tilde{L} _{(r,A,p)}:=\left\{ \left( \begin{array}{ccc} \check{x} \\ \check{y} \end{array} \right)\in \mathbb{R}^{2n} \ | \ \check{y}=\frac{1}{r}A\check{x}+\frac{1}{r}p \right\}.
\end{equation*}
We see that this $n$-dimensional submanifold $\tilde{L} _{(r,A,p)}$ satisfies the property (\ref{f1}), namely, $\tilde{L}_{(r,A,p)}$ becomes a Lagrangian submanifold in $\mathbb{R}^{2n}$ if and only if $\omega A=(\omega A)^t$ holds. Then, for the covering map $\pi : \mathbb{R}^{2n}\rightarrow \check{T}^{2n}_{J=T}$, 
\begin{equation*}
L_{(r,A,p)}:=\pi (\tilde{L}_{(r,A,p)})
\end{equation*}
defines a Lagrangian submanifold in $\check{T}^{2n}_{J=T}$. On the other hand, we can also regard the complexified symplectic torus $\check{T}^{2n}_{J=T}$ as the trivial special Lagrangian torus fibration $\check{\pi } : \check{T}^{2n}_{J=T} \rightarrow \mathbb{R}^n/2\pi \mathbb{Z}^n$, where $\check{x}$ is the local coordinates of the base space $\mathbb{R}^n/2\pi \mathbb{Z}^n$ and $\check{y}$ is the local coordinates of the fiber of $\check{\pi } : \check{T}^{2n}_{J=T} \rightarrow \mathbb{R}^n/2\pi \mathbb{Z}^n$. Then, we can interpret each affine Lagrangian submanifold $L_{(r,A,p)}$ in $\check{T}^{2n}_{J=T}$ as the affine Lagrangian multi section 
\begin{equation*}
s(\check{x})=\frac{1}{r}A\check{x}+\frac{1}{r}p
\end{equation*}
of $\check{\pi } : \check{T}^{2n}_{J=T} \rightarrow \mathbb{R}^n/2\pi \mathbb{Z}^n$.
\begin{rem}
As explained in section 3, while $r':=r_1'\cdots r_s'\in \mathbb{N}$ is the rank of $E_{(r,A,r',\mathcal{U},p,q)}\rightarrow T^{2n}_{J=T}$ $($see the relations $(\ref{matAB})$ and $(\ref{r'}))$, in the symplectic geometry side, this $r'\in \mathbb{N}$ is interpreted as follows. For the affine Lagrangian submanifold $L_{(r,A,p)}$ in $\check{T}^{2n}_{J=T}$ which is defined by a given data $(r,A,p)\in \mathbb{N}\times M(n;\mathbb{Z})\times \mathbb{R}^n$, we regard it as the affine Lagrangian multi section $s(\check{x})=\frac{1}{r}A\check{x}+\frac{1}{r}p$ of $\check{\pi } : \check{T}^{2n}_{J=T} \rightarrow \mathbb{R}^n/2\pi \mathbb{Z}^n$. Then, for each point $\check{x}\in \mathbb{R}^n/2\pi \mathbb{Z}^n$, we see
\begin{align*}
s(\check{x})=\biggl\{ & \left( \frac{1}{r}A\check{x}+\frac{1}{r}p+\frac{2\pi }{r}A\mathcal{B}M_s \right) \in \check{\pi }^{-1}(\check{x})\approx \mathbb{R}^n/2\pi \mathbb{Z}^n \ | \\
&M_s=(m_1,\cdots, m_s, 0,\cdots, 0)^t\in \mathbb{Z}^n, \ 0\leq m_i\leq r_i'-1, \ i=1,\cdots, s \biggr\},
\end{align*}
and this indicates that $s(\check{x})$ consists of $r'$ points. Thus, we can regard $r'\in \mathbb{N}$ as the multiplicity of $s(\check{x})=\frac{1}{r}A\check{x}+\frac{1}{r}p$.
\end{rem}
We then consider the trivial complex line bundle 
\begin{equation*}
\mathcal{L}_{(r,A,p,q)}\rightarrow L_{(r,A,p)}
\end{equation*}
with the flat connection
\begin{equation*}
\nabla_{\mathcal{L}_{(r,A,p,q)}}:=d-\frac{\mathbf{i}}{2\pi }\frac{1}{r}q^t d\check{x},
\end{equation*}
where $q\in \mathbb{R}^n$ is the unitary holonomy of $\mathcal{L}_{(r,A,p,q)}$ along $L_{(r,A,p)}\approx T^n$. We discuss the property (\ref{f2}) for this pair $(L_{(r,A,p)}, \mathcal{L}_{(r,A,p,q)})$ :
\begin{equation*}
\Omega _{\mathcal{L}_{(r,A,p,q)}}=\left. d\check{x}^t B d\check{y} \right|_{L_{(r,A,p)}}.
\end{equation*}
Here, $\Omega _{\mathcal{L}_{(r,A,p,q)}}$ is the curvature form of the flat connection $\nabla_{\mathcal{L}_{(r,A,p,q)}}$, i.e., $\Omega _{\mathcal{L}_{(r,A,p,q)}}=0$. Hence, we see 
\begin{equation*}
\left. d\check{x}^t B d\check{y}\right|_{L_{(r,A,p)}}=\frac{1}{r}d\check{x}^t BA d\check{x}=0,
\end{equation*}
so one has $BA=(BA)^t$. Note that $\omega A=(\omega A)^t$ and $BA=(BA)^t$ hold if and only if $AT=(AT)^t$ holds. Hereafter, for simplicity, we set
\begin{equation*}
\mathscr{L}_{(r,A,p,q)}:=(L_{(r,A,p)},\mathcal{L}_{(r,A,p,q)}).
\end{equation*}
By summarizing the above discussions, we obtain the following proposition. In particular, the condition $AT=(AT)^t$ in the following proposition is also the condition such that a complex vector bundle $E_{(r,A,r',\mathcal{U},p,q)}\rightarrow T^{2n}_{J=T}$ becomes a holomorphic vector bundle (see Proposition \ref{prophol}).
\begin{proposition} \label{propfukob}
For a given quadruple $(r,A,p,q)\in \mathbb{N}\times M(n;\mathbb{Z})\times \mathbb{R}^n\times \mathbb{R}^n$, $\mathscr{L}_{(r,A,p,q)}$ gives an object of the Fukaya category on $\check{T}^{2n}_{J=T}$ if and only if $AT=(AT)^t$ holds.
\end{proposition}
\begin{definition}
We denote the full subcategory of the Fukaya category on $\check{T}^{2n}_{J=T}$ consisting of objects $\mathscr{L}_{(r,A,p,q)}$ which satisfy the condition $AT=(AT)^t$ by $Fuk_{\mathrm{aff}}(\check{T}^{2n}_{J=T})$.
\end{definition}

\subsection{The isomorphism classes of $\mathscr{L}_{(r,A,p,q)}$}
The discussions in this subsection correspond to the discussions in subsection 3.2, so throughout this subsection, we fix $r\in \mathbb{N}$, $A\in M(n;\mathbb{Z})$, and consider the condition such that $\mathscr{L}_{(r,A,p,q)}\cong \mathscr{L}_{(r,A,p',q')}$ holds, where
\begin{equation*}
p, q, p', q'\in \mathbb{R}^n.
\end{equation*}

We explain the definition of the equivalence of objects of the Fukaya categories (cf. \cite[Definition 1.4]{Fuk}). Let us consider a pair $(M,\Omega )$ consisting of an even dimensional differentiable manifold $M$ and a complexified symplectic form $\Omega$. We take two objects $\mathscr{L}:=(L,\mathcal{L})$, $\mathscr{L}':=(L',\mathcal{L}')$ of the Fukaya category $Fuk(M,\Omega )$, where $L$, $L'$ are Lagrangian submanifolds in $(M,\Omega )$, and $\mathcal{L}\rightarrow L$, $\mathcal{L}'\rightarrow L'$ are unitary local systems. Then, if there exists a symplectic automorphism $\Phi : (M,\Omega ) \stackrel{\sim }{\rightarrow } (M,\Omega )$ such that
\begin{align*}
&\Phi ^{-1}(L')=L, \\
&\Phi ^*\mathcal{L}'\cong \mathcal{L}, 
\end{align*}
we say that $\mathscr{L}$ is isomorphic to $\mathscr{L}'$, and write $\mathscr{L}\cong \mathscr{L}'$.

We consider the isomorphism classes of $\mathscr{L}_{(r,A,p,q)}$, namely, we consider the condition such that $\mathscr{L}_{(r,A,p,q)}\cong \mathscr{L}_{(r,A,p',q')}$ holds as an analogue of Theorem \ref{EE'}. Actually, the following theorem holds.
\begin{theo} \label{LL'}
Two objects $\mathscr{L}_{(r,A,p,q)}$, $\mathscr{L}_{(r,A,p',q')}$ are isomorphic to each other, 
\begin{equation*}
\mathscr{L}_{(r,A,p,q)}\cong \mathscr{L}_{(r,A,p',q')},
\end{equation*}
if and only if 
\begin{equation*}
p\equiv p' \ (\mathrm{mod} \ 2\pi r\mathcal{A}^{-1} \left( \begin{array}{ccccc} \frac{\mathbb{Z}}{r_1'} \\ \vdots \\ \frac{\mathbb{Z}}{r_s'} \\ \mathbb{Z} \\ \vdots \\ \mathbb{Z} \end{array} \right)), \ q\equiv q' \ (\mathrm{mod} \ 2\pi r(\mathcal{B}^{-1})^t \left( \begin{array}{ccccc} \frac{\mathbb{Z}}{r_1'} \\ \vdots \\ \frac{\mathbb{Z}}{r_s'} \\ \mathbb{Z} \\ \vdots \\ \mathbb{Z} \end{array} \right))
\end{equation*}
hold.
\end{theo}
\begin{proof}
First, in order to prove the statement of this theorem, we prepare some notations. Since we considered the complex torus $T^{2n}_{J=T'}=\mathbb{C}^n/2\pi (\mathbb{Z}^n \oplus \mathcal{B}^{-1}T\mathcal{A}^t \mathbb{Z}^n)$ which is biholomorphic to the complex torus $T^{2n}_{J=T}$ in Theorem \ref{EE'} (and Proposition \ref{simplicity}), we take a mirror partner of the complex torus $T^{2n}_{J=T'}$. Here, we consider the complexified symplectic torus $\check{T}_{J=T'}^{2n}:=T_{\tilde{\omega }=(-T'^{-1})^t}^{2n}$ as a mirror partner of the complex torus $T^{2n}_{J=T'}$. We denote the local coordinates of $\check{T}^{2n}_{J=T'}$ by $(X^1,\cdots, X^n,Y^1,\cdots, Y^n)^t$, and set
\begin{equation*}
\check{X}:=(X^1,\cdots, X^n)^t, \ \check{Y}:=(Y^1,\cdots, Y^n)^t.
\end{equation*}
Let us consider a diffeomorphism $\phi : \check{T}^{2n}_{J=T} \stackrel{\sim }{\rightarrow } \check{T}^{2n}_{J=T'}$ which is expressed locally as
\begin{equation*}
\left( \begin{array}{ccc} \check{X} \\ \check{Y} \end{array} \right)=\phi \left( \begin{array}{ccc} \check{x} \\ \check{y} \end{array} \right) = \left( \begin{array}{ccc} \mathcal{B}^{-1} & O \\ O & \mathcal{A} \end{array} \right) \left( \begin{array}{ccc} \check{x} \\ \check{y} \end{array} \right).
\end{equation*}
By a direct calculation, we see
\begin{equation*}
\phi ^*(d\check{X}^t (-T'^{-1})^t d\check{Y})=d\check{x}^t (-T^{-1})^t d\check{y},
\end{equation*}
where $d\check{X}:=(dX^1,\cdots, dX^n)^t$, $d\check{Y}:=(dY^1,\cdots, dY^n)^t$, so this diffeomorphism $\phi $ is a symplectomorphism. 

Now, we define
\begin{align*}
&(\phi ^{-1})^*\mathscr{L}_{(r,A,p,q)}:=((\phi ^{-1})^{-1}(L_{(r,A,p)}), (\phi ^{-1})^*\mathcal{L}_{(r,A,p,q)}), \\
&(\phi ^{-1})^*\mathscr{L}_{(r,A,p',q')}:=((\phi ^{-1})^{-1}(L_{(r,A,p')}), (\phi ^{-1})^*\mathcal{L}_{(r,A,p',q')}),
\end{align*}
and let us consider the condition such that
\begin{equation*}
(\phi ^{-1})^*\mathscr{L}_{(r,A,p,q)}\cong (\phi ^{-1})^*\mathscr{L}_{(r,A,p',q')}
\end{equation*}
holds. Namely, our first goal is to consider when it is possible to construct a symplectic automorphism $\Phi : \check{T}^{2n}_{J=T'} \stackrel{\sim }{\rightarrow } \check{T}^{2n}_{J=T'}$ such that
\begin{align}
&\Phi ^{-1}((\phi ^{-1})^{-1}(L_{(r,A,p')}))=(\phi ^{-1})^{-1}(L_{(r,A,p)}), \label{sympauto1} \\
&\Phi ^* (\phi ^{-1})^*\mathcal{L}_{(r,A,p',q')} \cong (\phi ^{-1})^*\mathcal{L}_{(r,A,p,q)}. \label{sympauto2}
\end{align}
We consider the condition (\ref{sympauto1}). Since $r\in \mathbb{N}$, $A\in M(n;\mathbb{Z})$ are fixed, we can take the map 
\begin{equation*}
\Phi =\mathrm{id}_{\check{T}^{2n}_{J=T'}}
\end{equation*}
as a symplectic automorphism $\Phi : \check{T}^{2n}_{J=T'} \stackrel{\sim }{\rightarrow } \check{T}^{2n}_{J=T'}$ which satisfies the condition (\ref{sympauto1}), in the case $(\phi ^{-1})^{-1}(L_{(r,A,p)})=(\phi ^{-1})^{-1}(L_{(r,A,p')})$ only. Moreover,
\begin{equation*}
(\phi ^{-1})^{-1}(L_{(r,A,p)})=(\phi ^{-1})^{-1}(L_{(r,A,p')})
\end{equation*}
holds if and only if
\begin{equation}
\mathcal{A}p\equiv \mathcal{A}p' \ (\mathrm{mod} \ 2\pi r\left( \begin{array}{ccccc} \frac{\mathbb{Z}}{r_1'} \\ \vdots \\ \frac{\mathbb{Z}}{r_s'} \\ \mathbb{Z} \\ \vdots \\ \mathbb{Z} \end{array} \right)) \label{p}
\end{equation}
holds. Then, the condition (\ref{sympauto2}) becomes 
\begin{equation*}
(\phi ^{-1})^*\mathcal{L}_{(r,A,p,q)}\cong (\phi ^{-1})^*\mathcal{L}_{(r,A,p',q')},
\end{equation*}
so hereafter, we consider when $(\phi ^{-1})^*\mathcal{L}_{(r,A,p,q)}\cong (\phi ^{-1})^*\mathcal{L}_{(r,A,p',q')}$ holds on $(\phi ^{-1})^{-1}(L_{(r,A,p)})=(\phi ^{-1})^{-1}(L_{(r,A,p')})$ by computing an isomorphism
\begin{equation*}
\psi : (\phi ^{-1})^*\mathcal{L}_{(r,A,p,q)} \stackrel{\sim }{\rightarrow } (\phi ^{-1})^*\mathcal{L}_{(r,A,p',q')}
\end{equation*}
explicitly. The morphism $\psi $ need to satisfy the differential equation
\begin{equation}
\nabla_{(\phi ^{-1})^*\mathcal{L}_{(r,A,p',q')}}\psi =\psi \nabla_{(\phi ^{-1})^*\mathcal{L}_{(r,A,p,q)}}. \label{de}
\end{equation}
In particular, since the differential equation (\ref{de}) turns out to be
\begin{equation}
\left( \frac{\partial \psi }{\partial X^1},\cdots, \frac{\partial \psi }{\partial X^n} \right) -\frac{\mathbf{i}}{2\pi r}((q'^t\mathcal{B})_1-(q^t\mathcal{B})_1,\cdots, (q'^t\mathcal{B})_n-(q^t\mathcal{B})_n)\psi =0, \label{de'}
\end{equation}
by solving the differential equation (\ref{de'}), we obtain a solution
\begin{equation}
\psi (\check{X})=\lambda e^{\frac{\mathbf{i}}{2\pi r}(q'-q)^t\mathcal{B}\check{X}}, \label{psi}
\end{equation}
where $\lambda \in \mathbb{C}$ is an arbitrary constant. Furthermore, since $(\phi ^{-1})^*\mathcal{L}_{(r,A,p,q)}$ and $(\phi ^{-1})^*\mathcal{L}_{(r,A,p',q')}$ are trivial, this morphism $\psi $ must satisfy 
\begin{align*}
&\psi (X^1,\cdots, X^i+2\pi r_i',\cdots, X^n)=\psi (X^1,\cdots, X^i,\cdots, X^n) \ (i=1,\cdots, s), \\
&\psi (X^1,\cdots, X^i+2\pi ,\cdots, X^n)=\psi (X^1,\cdots, X^i,\cdots, X^n) \ (i=s+1,\cdots, n).
\end{align*}
By a direct calculation, we see that
\begin{align*}
&\psi (X^1,\cdots, X^i+2\pi r_i',\cdots, X^n)=e^{\frac{r_i'}{r}\mathbf{i}(q'-q)^t\mathcal{B}_i}\psi (X^1,\cdots, X^i,\cdots, X^n) \ (i=1,\cdots, s), \\
&\psi (X^1,\cdots, X^i+2\pi ,\cdots, X^n)=e^{\frac{\mathbf{i}}{r}(q'-q)^t\mathcal{B}_i}\psi (X^1,\cdots, X^i,\cdots, X^n) \ (i=s+1,\cdots, n)
\end{align*}
hold, where $\mathcal{B}_i:=(\mathcal{B}_{1i},\cdots, \mathcal{B}_{ni})^t$, so we obtain 
\begin{align*}
e^{\frac{r_i'}{r}\mathbf{i}(q'-q)^t\mathcal{B}_i}=1 \ (i=1,\cdots, s), \\\
e^{\frac{\mathbf{i}}{r}(q'-q)^t\mathcal{B}_i}=1 \ (i=s+1,\cdots, n).
\end{align*}
Clearly, these relations are equivalent to 
\begin{equation*}
\mathcal{B}^t(q'-q)=2\pi rN_{r'},
\end{equation*}
where
\begin{equation*}
N_{r'}:=\left( \frac{N_1}{r_1'},\cdots, \frac{N_s}{r_s'},N_{s+1},\cdots, N_n \right)^t, \ (N_1,\cdots, N_s, N_{s+1},\cdots, N_n)\in \mathbb{Z}^n,
\end{equation*}
and then, by the formula (\ref{psi}), the isomorphism $\psi $ is expressed locally as
\begin{equation*}
\psi (\check{X})=\lambda e^{\mathbf{i}N_{r'}^t\check{X}}
\end{equation*}
with $\lambda \not=0\in \mathbb{C}$. Hence, 
\begin{equation*}
(\phi ^{-1})^*\mathcal{L}_{(r,A,p,q)}\cong (\phi ^{-1})^*\mathcal{L}_{(r,A,p',q')}
\end{equation*}
holds on $(\phi ^{-1})^{-1}(L_{(r,A,p)})=(\phi ^{-1})^{-1}(L_{(r,A,p')})$ if and only if
\begin{equation}
\mathcal{B}^t q\equiv \mathcal{B}^t q' \ (\mathrm{mod} \ 2\pi r\left( \begin{array}{ccccc} \frac{\mathbb{Z}}{r_1'} \\ \vdots \\ \frac{\mathbb{Z}}{r_s'} \\ \mathbb{Z} \\ \vdots \\ \mathbb{Z} \end{array} \right)) \label{q}
\end{equation}
holds. Thus, by the relations (\ref{p}) and (\ref{q}), we can conclude that 
\begin{equation*}
\mathscr{L}_{(r,A,p,q)}\cong \mathscr{L}_{(r,A,p',q')}
\end{equation*}
holds if and only if
\begin{equation*}
p\equiv p' \ (\mathrm{mod} \ 2\pi r\mathcal{A}^{-1} \left( \begin{array}{ccccc} \frac{\mathbb{Z}}{r_1'} \\ \vdots \\ \frac{\mathbb{Z}}{r_s'} \\ \mathbb{Z} \\ \vdots \\ \mathbb{Z} \end{array} \right)), \ q\equiv q' \ (\mathrm{mod} \ 2\pi r(\mathcal{B}^{-1})^t \left( \begin{array}{ccccc} \frac{\mathbb{Z}}{r_1'} \\ \vdots \\ \frac{\mathbb{Z}}{r_s'} \\ \mathbb{Z} \\ \vdots \\ \mathbb{Z} \end{array} \right))
\end{equation*}
hold.
\end{proof}
Hence, by comparing Theorem \ref{EE'} with Theorem \ref{LL'}, we can expect that the isomorphism classes of holomorphic vector bundles $E_{(r,A,r',\mathcal{U},p,q)}\rightarrow T^{2n}_{J=T}$ correspond to the isomorphism classes of objects $\mathscr{L}_{(r,A,p,q)}$ of the Fukaya category $Fuk_{\mathrm{aff}}(\check{T}^{2n}_{J=T})$. Actually, by a direct calculation, we can check that there exists such a correspondence.

\section{Main result}
In this section, we prove that there exists a bijection between the set of the isomorphism classes of holomorphic vector bundles $E_{(r,A,r',\mathcal{U},p,q)}\rightarrow T^{2n}_{J=T}$ and the set of the isomorphism classes of objects $\mathscr{L}_{(r,A,p,q)}$ of the Fukaya category $Fuk_{\mathrm{aff}}(\check{T}^{2n}_{J=T})$.

First, we prepare two notations. We denote the set of the isomorphism classes of objects of the DG-category $DG_{T^{2n}_{J=T}}$ (i.e., the set of the isomorphism classes of holomorphic vector bundles $E_{(r,A,r',\mathcal{U},p,q)}$) by
\begin{equation*}
\mathrm{Ob}^{isom}(DG_{T^{2n}_{J=T}}).
\end{equation*}
Similarly, we denote the set of the isomorphism classes of objects $\mathscr{L}_{(r,A,p,q)}$ of the Fukaya category $Fuk_{\mathrm{aff}}(\check{T}^{2n}_{J=T})$ by
\begin{equation*}
\mathrm{Ob}^{isom}(Fuk_{\mathrm{aff}}(\check{T}^{2n}_{J=T})).
\end{equation*}

Now, in order to state the main theorem, we define a map $F : \mathrm{Ob}(DG_{T^{2n}_{J=T}})\rightarrow \mathrm{Ob}(Fuk_{\mathrm{aff}}(\check{T}^{2n}_{J=T}))$ as follows. Clearly, we need four parameters $r$, $A$, $p$, $q$ when we define objects $\mathscr{L}_{(r,A,p,q)}$ of $Fuk_{\mathrm{aff}}(\check{T}^{2n}_{J=T})$. On the contrary, we need five parameters $r$, $A$, $p$, $q$, $\mathcal{U}$ when we define objects $E_{(r,A,r',\mathcal{U},p,q)}$ of $DG_{T^{2n}_{J=T}}$. Hence, when we define a map $\mathrm{Ob}(DG_{T^{2n}_{J=T}})\rightarrow \mathrm{Ob}(Fuk_{\mathrm{aff}}(\check{T}^{2n}_{J=T}))$, we must transform not only the information about four parameters $r$, $A$, $p$, $q$ but also the information about $\mathcal{U}$. For example, let us consider a map $\mathrm{Ob}(DG_{T^{2n}_{J=T}})\rightarrow \mathrm{Ob}(Fuk_{\mathrm{aff}}(\check{T}^{2n}_{J=T}))$ which is simply defined by
\begin{equation*}
E_{(r,A,r',\mathcal{U},p,q)}\mapsto \mathscr{L}_{(r,A,p,q)}.
\end{equation*}
Then, this map does not induce a bijection between $\mathrm{Ob}^{isom}(DG_{T^{2n}_{J=T}})$ and $\mathrm{Ob}^{isom}(Fuk_{\mathrm{aff}}(\check{T}^{2n}_{J=T}))$ unfortunately, and we can check it as follows. We set
\begin{equation*}
T=\mathbf{i}\cdot I_2, \ r=2, \ A=\left( \begin{array}{ccc} 1 & 0 \\ 0 & 0 \end{array} \right), \ p=q=0.
\end{equation*}
It is clear that $AT=(AT)^t$ holds and $r'=2$. For this quadruple $(r,A,p,q)\in \mathbb{N} \times M(2;\mathbb{Z}) \times \mathbb{R}^2 \times \mathbb{R}^2$, we define mutually distinct $\mathcal{U}$ and $\mathcal{U}'$ by
\begin{align*}
&\mathcal{U}:=\left\{ V_1=\left( \begin{array}{ccc} 0 & 1 \\ 1 & 0 \end{array} \right), \ U_1=\left( \begin{array}{ccc} 1 & 0 \\ 0 & -1 \end{array} \right), \ V_2=U_2=I_2 \in U(2) \right\}, \\
&\mathcal{U'}:=\left\{ V'_1=\left( \begin{array}{ccc} 0 & 1 \\ 1 & 0 \end{array} \right), \ U'_1=\left( \begin{array}{ccc} \mathbf{i} & 0 \\ 0 & -\mathbf{i} \end{array} \right), \ V'_2=U'_2=I_2 \in U(2) \right\}.
\end{align*}
In this situation, we can easily check that $E_{(r,A,r',\mathcal{U},p,q)}\not \cong E_{(r,A,r',\mathcal{U'},p,q)}$ holds by using Theorem \ref{EE'}, and actually, the above map $E_{(r,A,r',\mathcal{U},p,q)}\mapsto \mathscr{L}_{(r,A,p,q)}$ sends both $E_{(r,A,r',\mathcal{U},p,q)}$ and $E_{(r,A,r',\mathcal{U}',p,q)}$ to the same object $\mathscr{L}_{(r,A,p,q)}$. Thus, by concerning these facts, here, we define a map 
\begin{equation*}
F : \mathrm{Ob}(DG_{T^{2n}_{J=T}})\rightarrow \mathrm{Ob}(Fuk_{\mathrm{aff}}(\check{T}^{2n}_{J=T}))
\end{equation*}
by
\begin{equation*}
F(E_{(r,A,r',\mathcal{U},p,q)})=\mathscr{L}_{(r,A,p-\frac{r}{r'}\theta ,q+\frac{r}{r'}\xi )},
\end{equation*}
where $\xi $, $\theta \in \mathbb{R}^n$ denote the vectors associated to $\mathcal{U}$ in the sense of the definition (\ref{xitheta}). The following is the main theorem of this paper.
\begin{theo} \label{bijectivity}
The map $F$ induces a bijection between $\mathrm{Ob}^{isom}(DG_{T^{2n}_{J=T}})$ and $\mathrm{Ob}^{isom}(Fuk_{\mathrm{aff}}(\check{T}_{J=T}^{2n}))$.
\end{theo}
\begin{proof}
In this proof, for a given object $E_{(r,A,r',\mathcal{U},p,q)}\in \mathrm{Ob}(DG_{T^{2n}_{J=T}})$, we denote by $\xi $, $\theta \in \mathbb{R}^n$ the vectors associated to $\mathcal{U}$ in the sense of the definition (\ref{xitheta}). Similarly, for a given object $E_{(s,B,s',\mathcal{V},u,v)}\in \mathrm{Ob}(DG_{T^{2n}_{J=T}})$, we denote by $\tau $, $\sigma \in \mathbb{R}^n$ the vectors associated to $\mathcal{V}$ in the sense of the definition (\ref{xitheta}). We denote the induced map from the map $F$ by
\begin{equation*}
F^{isom} : \mathrm{Ob}^{isom}(DG_{T^{2n}_{J=T}}) \rightarrow \mathrm{Ob}^{isom}(Fuk_{\mathrm{aff}}(\check{T}^{2n}_{J=T})).
\end{equation*}
Explicitly, it is defined by
\begin{equation*}
F^{isom}([E_{(r,A,r',\mathcal{U},p,q)}])=[F(E_{(r,A,r',\mathcal{U},p,q)})],
\end{equation*}
where, of course, $[E_{(r,A,r',\mathcal{U},p,q)}]$ and $[F(E_{(r,A,r',\mathcal{U},p,q)})]$ denote the isomorphism class of $E_{(r,A,r',\mathcal{U},p,q)}$ and the isomorphism class of $F(E_{(r,A,r',\mathcal{U},p,q)})$, respectively. 

First, we check the well-definedness of the map $F^{isom}$. We take two arbitrary objects $E_{(r,A,r',\mathcal{U},p,q)}$, $E_{(s,B,s',\mathcal{V},u,v)}\in \mathrm{Ob}(DG_{T^{2n}_{J=T}})$ and assume 
\begin{equation*}
E_{(r,A,r',\mathcal{U},p,q)}\cong E_{(s,B,s',\mathcal{V},u,v)}.
\end{equation*}
By considering the $i$-th Chern characters $ch_i(E_{(r,A,r',\mathcal{U},p,q)})$, $ch_i(E_{(s,B,s',\mathcal{V},u,v)})$ of the holomorphic vector bundles $E_{(r,A,r',\mathcal{U},p,q)}$, $E_{(s,B,s',\mathcal{V},u,v)}$ for each $i\in \mathbb{N}$, we see
\begin{equation}
ch_i(E_{(r,A,r',\mathcal{U},p,q)}) = ch_i(E_{(s,B,s',\mathcal{V},u,v)}). \label{chern}
\end{equation}
We consider the equality (\ref{chern}) in the cases $i=0,1$. Then, we obtain
\begin{align*}
&r'=s', \\
&\frac{r'}{r}A=\frac{s'}{s}B, 
\end{align*}
so one has
\begin{equation*}
\frac{1}{r}A=\frac{1}{s}B. 
\end{equation*}
Hence, we see that there exists a $k\in \mathbb{N}$ such that
\begin{align}
&s=kr, \label{s} \\
&B=kA. \label{B}
\end{align}
Therefore, since we can regard $E_{(s,B,s',\mathcal{V},u,v)}$ as $E_{(kr,kA,r',\mathcal{V},u,v)}=E_{(r,A,r',\mathcal{V},\frac{1}{k}u,\frac{1}{k}v)}$, by Theorem \ref{EE'}, we see
\begin{align}
&p-\frac{r}{r'}\theta \equiv \frac{1}{k}\left( u-\frac{s}{s'}\sigma \right) \ (\mathrm{mod} \ 2\pi r\mathcal{A}^{-1} \left( \begin{array}{ccccc} \frac{\mathbb{Z}}{r_1'} \\ \vdots \\ \frac{\mathbb{Z}}{r_s'} \\ \mathbb{Z} \\ \vdots \\ \mathbb{Z} \end{array} \right) ), \label{pp'} \\
&q+\frac{r}{r'}\xi \equiv \frac{1}{k} \left( v+\frac{s}{s'}\tau \right) \ (\mathrm{mod} \ 2\pi r(\mathcal{B}^{-1})^t \left( \begin{array}{ccccc} \frac{\mathbb{Z}}{r_1'} \\ \vdots \\ \frac{\mathbb{Z}}{r_s'} \\ \mathbb{Z} \\ \vdots \\ \mathbb{Z} \end{array} \right) ). \label{qq'}
\end{align}
Thus, by Theorem \ref{LL'} and the relations (\ref{s}), (\ref{B}), (\ref{pp'}), (\ref{qq'}), we can conclude
\begin{equation*}
\mathscr{L}_{(r,A,p-\frac{r}{r'}\theta ,q+\frac{r}{r'}\xi )}\cong \mathscr{L}_{(s,B,u-\frac{s}{s'}\sigma ,v+\frac{s}{s'}\tau )},
\end{equation*}
namely,
\begin{equation*}
F(E_{(r,A,r',\mathcal{U},p,q)})\cong F(E_{(s,B,s',\mathcal{V},u,v)}).
\end{equation*}

Next, we prove that $F^{isom}$ is injective. We take two arbitrary objects $E_{(r,A,r',\mathcal{U},p,q)}$, $E_{(s,B,s',\mathcal{V},u,v)}\in \mathrm{Ob}(DG_{T^{2n}_{J=T}})$ and assume
\begin{equation*}
F(E_{(r,A,r',\mathcal{U},p,q)})\cong F(E_{(s,B,s',\mathcal{V},u,v)}),
\end{equation*}
namely,
\begin{equation*}
\mathscr{L}_{(r,A,p-\frac{r}{r'}\theta ,q+\frac{r}{r'}\xi )}\cong \mathscr{L}_{(s,B,u-\frac{s}{s'}\sigma ,v+\frac{s}{s'}\tau )}.
\end{equation*}
Then, we see that there exists a $k\in \mathbb{N}$ which satisfies the relations (\ref{s}) and (\ref{B}). Here, we take two matrices $\mathcal{A}$, $\mathcal{B}\in GL(n;\mathbb{Z})$ such that
\begin{equation}
\mathcal{A}A\mathcal{B}=\left( \begin{array}{ccccccc} \tilde{a_1} & & & & & \\ & \ddots & & & & \\ & & \tilde{a_t} & & & \\ & & & 0 & & \\ & & & & \ddots & \\ & & & & & 0 \end{array} \right), \label{mat1}
\end{equation}
where $\tilde{a_i}\in \mathbb{N}$ ($i=1,\cdots, t$, $1\leq t\leq n$) and $\tilde{a_i} | \tilde{a_{i+1}}$ ($i=1,\cdots, t-1$). Therefore, since the relation (\ref{B}) holds, we obtain
\begin{equation}
\mathcal{A}B\mathcal{B}=\mathcal{A} ( kA ) \mathcal{B}=\left( \begin{array}{ccccccc} k\tilde{a_1} & & & & & \\ & \ddots & & & & \\ & & k\tilde{a_t} & & & \\ & & & 0 & & \\ & & & & \ddots & \\ & & & & & 0 \end{array} \right). \label{mat2}
\end{equation}
In particular, the relations (\ref{s}), (\ref{B}), (\ref{mat1}), (\ref{mat2}) imply
\begin{equation}
r'=s'. \label{eq10}
\end{equation}
Hence, by Theorem \ref{LL'}, the relations (\ref{pp'}) and (\ref{qq'}) hold. Now, note that we can regard $E_{(s,B,s',\mathcal{V},u,v)}$ as $E_{(kr,kA,r',\mathcal{V},u,v)}=E_{(r,A,r',\mathcal{V},\frac{1}{k}u,\frac{1}{k}v)}$ by the relations (\ref{s}), (\ref{B}), (\ref{eq10}). Thus, by Theorem \ref{EE'} and the relations (\ref{pp'}) and (\ref{qq'}), we see that
\begin{equation*}
E_{(r,A,r',\mathcal{U},p,q)}\cong E_{(r,A,r',\mathcal{V},\frac{1}{k}u,\frac{1}{k}v)}
\end{equation*}
holds, and this relation indicates
\begin{equation*}
E_{(r,A,r',\mathcal{U},p,q)}\cong E_{(s,B,s',\mathcal{V},u,v)}.
\end{equation*}

Finally, we prove that $F^{isom}$ is surjective. We take an arbitrary quadruple $(r,A,p,q)\in \mathbb{N}\times M(n;\mathbb{Z})\times \mathbb{R}^n\times \mathbb{R}^n$, and consider the element
\begin{equation*}
[\mathscr{L}_{(r,A,p,q)}]\in \mathrm{Ob}^{isom}(Fuk_{\mathrm{aff}}(\check{T}^{2n}_{J=T})).
\end{equation*}
In particular, a representative of $[\mathscr{L}_{(r,A,p,q)}]$ is expressed as
\begin{equation*}
\mathscr{L}_{(r,A,p+2\pi r\mathcal{A}^{-1}M,q+2\pi r(\mathcal{B}^{-1})^tN)}
\end{equation*}
by using a pair 
\begin{equation*}
(M,N)\in \left( \begin{array}{ccccc} \frac{\mathbb{Z}}{r_1'} \\ \vdots \\ \frac{\mathbb{Z}}{r_s'} \\ \mathbb{Z} \\ \vdots \\ \mathbb{Z} \end{array} \right) \times \left( \begin{array}{ccccc} \frac{\mathbb{Z}}{r_1'} \\ \vdots \\ \frac{\mathbb{Z}}{r_s'} \\ \mathbb{Z} \\ \vdots \\ \mathbb{Z} \end{array} \right)
\end{equation*}
(see Theorem \ref{LL'}). For the element $[\mathscr{L}_{(r,A,p,q)}]$, we consider the element
\begin{equation*}
[E_{(r,A,r',\mathcal{U},p+\frac{r}{r'}\theta ,q-\frac{r}{r'}\xi )}]\in \mathrm{Ob}^{isom}(DG_{T^{2n}_{J=T}}),
\end{equation*}
where $\xi $, $\theta \in \mathbb{R}^n$ are the vectors associated to $\mathcal{U}$ in the sense of the definition (\ref{xitheta}). Here, note that how to choose a set $\mathcal{U}$ is not unique even if we fix a quadruple $(r,A,p,q)\in \mathbb{N}\times M(n;\mathbb{Z})\times \mathbb{R}^n\times \mathbb{R}^n$. Therefore, a representative of $[E_{(r,A,r',\mathcal{U},p+\frac{r}{r'}\theta ,q-\frac{r}{r'}\xi )}]$ is expressed as
\begin{equation*}
E_{(r,A,r',\mathcal{U}',p+\frac{r}{r'}\theta '+2\pi r\mathcal{A}^{-1}M,q-\frac{r}{r'}\xi '+2\pi r(\mathcal{B}^{-1})^tN)}
\end{equation*}
by using a set $\mathcal{U}'$ with the associated vectors $\xi '$, $\theta '\in \mathbb{R}^n$ and a pair 
\begin{equation*}
(M,N)\in \left( \begin{array}{ccccc} \frac{\mathbb{Z}}{r_1'} \\ \vdots \\ \frac{\mathbb{Z}}{r_s'} \\ \mathbb{Z} \\ \vdots \\ \mathbb{Z} \end{array} \right) \times \left( \begin{array}{ccccc} \frac{\mathbb{Z}}{r_1'} \\ \vdots \\ \frac{\mathbb{Z}}{r_s'} \\ \mathbb{Z} \\ \vdots \\ \mathbb{Z} \end{array} \right)
\end{equation*}
(see Theorem \ref{EE'}). Then, by a direct calculation, we see
\begin{align*}
&F^{isom}([E_{(r,A,r',\mathcal{U},p+\frac{r}{r'}\theta ,q-\frac{r}{r'}\xi )}]) \\
&=F^{isom}([E_{(r,A,r',\mathcal{U}',p+\frac{r}{r'}\theta '+2\pi r\mathcal{A}^{-1}M,q-\frac{r}{r'}\xi '+2\pi r(\mathcal{B}^{-1})^tN)}]) \\
&=[F(E_{(r,A,r',\mathcal{U}',p+\frac{r}{r'}\theta '+2\pi r\mathcal{A}^{-1}M,q-\frac{r}{r'}\xi '+2\pi r(\mathcal{B}^{-1})^tN)})] \\
&=[\mathscr{L}_{(r,A,(p+\frac{r}{r'}\theta '+2\pi r\mathcal{A}^{-1}M)-\frac{r}{r'}\theta ',(q-\frac{r}{r'}\xi '+2\pi r(\mathcal{B}^{-1})^tN)+\frac{r}{r'}\xi ')}] \\
&=[\mathscr{L}_{(r,A,p+2\pi r\mathcal{A}^{-1}M,q+2\pi r(\mathcal{B}^{-1})^tN)}] \\
&=[\mathscr{L}_{(r,A,p,q)}].
\end{align*}
This completes the proof.
\end{proof}

\section*{Acknowledgment}
I would like to thank Hiroshige Kajiura for various advices in writing this paper. I also would like to thank Masahiro Futaki and Atsushi Takahashi for helpful comments. Finally, I am grateful to the referee for useful suggestions. This work was supported by Grant-in-Aid for JSPS Research Fellow 18J10909.


\begin{thebibliography}{99}
\bibitem{abouzaid}
M. Abouzaid, I. Smith, Homological mirror symmetry for the four-torus, Duke Mathematical Journal, 152.3 (2010), 373-440.
\bibitem{A-P}
D. Arinkin, A. Polishchuk, Fukaya category and Fourier transform, AMS IP STUDIES IN ADVANCED MATHEMATICS, 2001, 23 : 261-274.
\bibitem{bondal}
A. Bondal and M. Kapranov, Enhanced triangulated categories, Math. USSR Sbornik 70:93-107, 1991.
\bibitem{Fukaya category}
K. Fukaya, Morse homotopy, $A^{\infty }$-category, and Floer homologies, In: Proceedings of GARC Workshop on Geometry and Topology '93 (Seoul, 1993). Lecture Notes in Series, vol. 18, pp. 1-102. Seoul Nat. Univ., Seoul (1993).
\bibitem{Fuk}
K. Fukaya, Mirror symmetry of abelian varieties and multi-theta functions, J. Alg. Geom. 11, 393-512 (2002).
\bibitem{kaji}
H. Kajiura, On some deformation of fukaya categories. Symplectic, Poisson, and Noncommutative Geometry, 93-130, MSRI Publ. 62, Cambridge Univ. Press, New York, 2014.
\bibitem{kazushi}
K. Kobayashi, On exact triangles consisting of stable vector bundles on tori, Differential Geometry and its Applications 53 (2017) 268-292, arXiv : mathDG/1610.02821.
\bibitem{kazushi2}
K. Kobayashi, Geometric interpretation for exact triangles consisting of projectively flat bundles on higher dimensional complex tori, arXiv : mathDG/1705.04007.
\bibitem{D}
K. Kobayashi, On exact triangles consisting of projectively flat bundles on complex tori, doctoral thesis at Chiba University.
\bibitem{kazushi3}
K. Kobayashi, Remarks on the homological mirror symmetry for tori, arXiv : mathDG/2004.05621.
\bibitem{koba}
S. Kobayashi, Differential Geometry of Complex Vector Bundles, Princeton University Press, 1987.
\bibitem{Kon}
M. Kontsevich, Homological algebra of mirror symmetry, In Proceedings of the International Congress of Mathematicians, Vol. 1, 2 (Z\"{u}rich, 1994), Birkh\"{a}user, 1995, pages 120-139, arXiv : math.AG/9411018.
\bibitem{dg}
M. Kontsevich, Y. Soibelman, Homological mirror symmetry and torus fibrations. In Symplectic geometry and mirror symmetry (Seoul, 2000), pages 203-263. World Sci.Publishing, River Edge, NJ, 2001. math.SG/0011041.
\bibitem{leung}
N. C. Leung, S.-T. Yau, E. Zaslow, From special Lagrangian to Hermitian-Yang-Mills via Fourier-Mukai transform, Adv. Theor. Math. Phys. 4:13191341, 2000.
\bibitem{matsu}
Y. Matsushima, Heisenberg groups and holomorphic vector bundles over a complex torus, Nagoya Math. J. Vol. 61 (1976), 161-195.
\bibitem{mukai}
S. Mukai, Semi-homogeneous vector bundles on an abelian variety, J. Math. Kyoto Univ. 18 (1978), no. 2, 239-272.
\bibitem{orlov}
D. O. Orlov, Remarks on generators and dimensions of triangulated categories, Moscow Mathematical Journal 9.1 (2009) : 143-149, arXiv : math.AG/0804.1163.
\bibitem{A-inf}
A. Polishchuk, $A_{\infty }$-structures on an elliptic curve, Comm. Math. Phys. 247, 527 (2004), arXiv : math.AG/0001048.
\bibitem{elliptic}
A. Polishchuk, E. Zaslow, Categorical mirror symmetry : the elliptic curve, Adv. Theor. Math. Phys. 2, 443-470 (1998), arXiv : math.AG/9801119.
\bibitem{SYZ}
A. Strominger, S.-T. Yau, and E. Zaslow, Mirror Symmetry is T-duality, Nucl. Phys. B, 479:243-259, 1996.
\end{thebibliography}
\end{document}